\documentclass[11pt]{amsart}

\usepackage{amsmath,amsfonts,amssymb,mathrsfs}
\usepackage{graphicx,extpfeil}
\usepackage{indentfirst, latexsym, enumerate}
\usepackage{hyperref,bookmark}
\usepackage{epsfig,subfigure}
\usepackage{colortbl,bm}

\title[Trend to equilibrium and diffusion limit for K-S]{Trend to equilibrium and diffusion limit for the inertial Kuramoto-Sakaguchi equation}

\author[Francis Filbet]{Francis Filbet}
\address[Francis Filbet]{Institut de Math\'ematiques de Toulouse, Universit\'e Paul Sabatier, Toulouse, France}
\email{francis.filbet@math.univ-toulouse.fr}

\author[Myeongju  Kang]{Myeongju Kang}
\address[Myeongju  Kang]{{Department of Finance and Big Data, Gacheon University, Seongnam, Republic of Korea}}
\email{{mathemjkang@gachon.ac.kr}}

\newtheorem{theorem}{Theorem}[section]
\newtheorem{lemma}{Lemma}[section]

\newtheorem{proposition}{Proposition}[section]

\newtheorem{remark}{Remark}[section]

\newcommand{\T} {\mathbb T}

\newcommand{\N} {\mathbb N}
\newcommand{\R}{\mathbb R}

\newcommand{\bz}{\mbox{\boldmath $z$}}
\newcommand{\dD}{\mathrm d}

\newcommand{\ds}{\displaystyle}
\newcommand{\cM}{\mathcal M}
\newcommand{\cL}{\mathcal L}
\newcommand{\cI}{\mathcal I}
\newcommand{\cT}{\mathcal T}
\newcommand{\cA}{\mathcal A}
\newcommand{\cE}{\mathcal E}
\newcommand{\myw}{\gamma}
\newcommand{\mybw}{{\bar\gamma}}
\newcommand{\tkappa}{\widetilde\kappa}
\newcommand{\tsigma}{\widetilde\sigma}
\newcommand{\eps}{\varepsilon}

\begin{document}

\begin{abstract}
  In this paper, we study the  inertial Kuramoto-Sakaguchi
  equation for  interacting oscillatory systems. On the one hand, we prove the convergence toward corresponding
  phase-homogeneous stationary states in weighted Lebesgue norm
  sense when the coupling strength is small enough. In \cite{CHXZ}, it
  is proved that when the noise intensity is sufficiently large, equilibrium of the inertial Kuramoto-Sakaguchi
  equation is asymptotically stable. For generic initial data, every
  solutions converges to equilibrium in weighted Sobolev norm
  sense. We improve this previous result by showing the convergence
  for a larger class of functions and by providing a simpler
  proof. On the other hand, we investigate the diffusion limit when all
  oscillators are identical. In \cite{HSZ}, authors studied the same
  problem using an energy estimate on {renormalized} solutions and a compactness method, through
  which error estimates could not be discussed. Here we provide
  error estimates for the diffusion limit with respect to the mass
  $m\ll 1$ using  a simple proof by  imposing
  slightly more regularity on the solution.
\end{abstract}

\date{\today}



\thanks{\textbf{Acknowledgment.} Both authors are grateful to
    Prof. S. Y. Ha for discussions on  the inertial Kuramoto-Sakaguchi
    equation. F. Filbet  would like to thank Seoul National University
    for hosting a very successful visit. {The work of M. Kang was supported by the Gachon University research fund of 2024.(GCU-202405190001).} All authors contributed equally to this work.}

\maketitle \centerline{\date}

\tableofcontents

%
%
%
%
%
%
%

\section{Introduction and main results}
\label{sec:Intro}
\setcounter{equation}{0}

Synchronous behavior of a large but loosely organized group of agents
is ubiquitously found in various social and biological phenomena, for
example, flashing of fireflies, beating of cardiac cells, and hand
clapping in opera, etc  \cite{ABPVRS, BB, Ermen, PRK, Strog, Wi1,
  Wi2}. Recently, collective dynamics of an interacting oscillatory
system has received more attention due to its diverse applications in
research areas of control theory, physics, neuroscience \cite{ABPVRS}.

Systematic studies on synchronization were invoked by the
pioneer works of  A. Winfree \cite{Wi1, Wi2} and Y. Kuramoto
\cite{Ku1, Ku2, Ku3}.  More precisely, these models consider a collection of $N \in \N$ oscillators,
represented by their phase-frequency pair
$(\theta^i,\omega^i)\in\mathbb T\times\mathbb R$  and by their natural
frequency $\nu^i$. In the presence of inertia and stochastic noise
effect, the dynamics of stochastic Kuramoto oscillators is governed by
the following set of globally coupled  ODEs \cite{ABPVRS, Ku3, Sa}:
\begin{align} \label{sys:Ku}
  \begin{dcases}
    \dD\theta^i_t \,=\, \omega^i_t \, \dD t, \\[.9em]
    m\, \dD\omega^i_t \,=\, \left( -\omega^i_t \,+\, \nu^i \,+\, \frac{\kappa}{N} \sum_{k=1}^N \sin\left( \theta^k_t-\theta^i_t \right) \right) \dD t \,+\, \sqrt{2\sigma} \, \dD B^i_t,
  \end{dcases}
\end{align}
where nonnegative coefficients $m$, $\kappa$, and $\sigma$ represent mass, coupling strength, and noise intensity, respectively, and $B^i_t$'s are independent one-dimensional Brownian motions. We refer to \cite{CDH, CHLXY, CHM, HKPZ, HR} for the emergent behavior of the Kuramoto model \eqref{sys:Ku} including zero inertia case ($m=0$) or noiseless case ($\sigma=0$). {It is noteworthy that one may consider the friction coefficient $\gamma>0$ by using $-\gamma\,\omega^i_t$ instead of $-\omega^i_t$ in \eqref{sys:Ku}. However, throughout this paper, we set $\gamma=1$, following the previous literature \cite{HSZ}, in order to facilitate comparison with earlier results. Consequently, some constants in this paper that appear dimensionless may, in fact, possess physical units.}

The continuum approximation assumes that populations at the
thermodynamic limit $N \rightarrow +\infty$ are described by
a continuous distribution $f = f(t, \theta, \omega, \nu)$  at
time $t\in\R^+$, phase $\theta\in\T$, frequency
$\omega\in\R$, and natural frequency $\nu\in\R$. Under
this assumption, the time evolution is governed by the following Vlasov-Fokker-Planck-type equation \cite{ABPVRS, HKPZ}:
\begin{align}
  \label{sys:KVFP}
  \begin{dcases}
    \partial_t f \,+\,\omega\, \partial_\theta f
    \,+\,\partial_\omega(f\,\cT[f]) \,=\,
    \frac{\sigma}{m^2}\,\partial_\omega^2 f,
    \\[0.9em]
    \cT[f](t, \bz) \,=\, -\frac{\omega}{m}
    \,+\,\frac{\nu}{m}\,+\,\frac{\kappa}{m}\,\int_{\T \times \R^2}
    \sin(\theta_*-\theta) f(t, \bz_*) \,\dD \bz_*,
    \\[0.9em]
    f(t=0) \,=\, f_{\rm in} \geq 0\,,
  \end{dcases}
\end{align}
where $\bz = (\theta, \omega, \nu)\in\T\times\R^2$. Observe that since
the $\nu$ variable only appears as a parameter we  have the following invariant
$$
\int_{\T\times \R} f(t, \bz) \,\dD\omega \dD\theta \,=\,
\int_{\T\times \R} f_{\rm in}(\bz) \,\dD\omega \dD\theta \,=:\, g(\nu).
$$
{
In addition, nonnegativity and mass of $f$ is conserved (see Proposition \ref{prop:basic}). Hence, without loss of generality, we assume $\|f_{\rm in}\|_{L^1} = 1$ throughout this paper.}

In order to highlight important parameters, we  introduce rescaled
coupling strength and noise intensity $\tkappa$ and $\tsigma$ as
\begin{equation*}
  \left\{
    \begin{array}{l}
      \ds\tkappa \,=\, \frac{\kappa}{m}\,,\\[0.9em]
      \ds\tsigma \,=\, \frac{\sigma}{m}
    \end{array}\right.
  \end{equation*}
and define the {phase density} as
\begin{align*}
  \rho(t, \theta) \,=\, \int_{\R^2} f(t, \bz) \, \dD\omega \,\dD\nu\,.
\end{align*}
Hence we  simplify $\cT[f]$ as
\begin{align*}
  \cT[f](t, \bz) \,=\, -\frac{\omega}{m} \,+\,\frac{\nu}{m}\,-\,\tkappa\, (\sin\ast\rho(t))(\theta).
\end{align*}
and system \eqref{sys:KVFP} can be written as
\begin{align} \label{K3-1}
  \begin{dcases}
  \partial_t f \,+\,\omega \,\partial_\theta f \,-\,\tkappa\,
  (\sin\ast\rho) \,\partial_\omega f \,=\, \cL_{\mathrm{FP}}[f]\,,
    \\[0.9em]
    f(t=0) \,=\, f_{\rm in} \geq 0\,,
  \end{dcases}
\end{align}
with
\begin{equation}
  \label{K3-2}
  \cL_{\mathrm{FP}}[f] \,:=\, \frac{1}{m} \,\partial_\omega \bigg( \tsigma\,\partial_\omega f \,+\,(\omega-\nu)\,f \bigg)\,.
\end{equation}
For the existence and uniqueness theory on \eqref{sys:KVFP}, we refer
to \cite{CHXZ, HSZ}, in which global nonnegative weak solutions are
constructed in weighted Sobolev spaces.

In this paper,  we first analyze the long time behavior of the
solution to  \eqref{sys:KVFP} in the regime where  the coupling strength $\tkappa>0$
is small compared to the noise intensity $\tsigma>0$. In this setting,
the global equilibrium is characterized by a spatially homogeneous
distribution and we prove that the solution of \eqref{sys:KVFP}
converges exponentially fast to this equilibrium, providing an explicit rate
of convergence. Second,  we focus on the  case of identical
oscillators and investigate the limit when the mass parameter $\eps:=m$ goes to
zero.  We prove that the local density $\rho^\eps$ converges to the solution of
the drift-diffusion equation. Our methods yields direct error estimates with
respect to $\eps$.

In the next subsections, we describe more precisely our main results and
the state of the art.

\subsection{Long time asymptotics}

First, note that nonnegative phase-homogeneous state {independent on the initial condition $f_{\infty} = f_{\infty}(\omega, \nu)$} becomes stationary solution of \eqref{K3-1}-\eqref{K3-2} if and only if
\begin{align*}
 {\cL_{\mathrm{FP}}[f_{\infty}] \,=\, 0 \quad \mbox{and} \quad 2\pi\,\int_{\R} f_{\infty}(\omega, \nu) \,\dD\omega \,=\, g(\nu)\,.}
\end{align*}
Thus we denote by  $N_\infty(\nu) \,=\, \rho_\infty \,g(\nu)$ with {$\rho_\infty\,=\, (2\pi)^{-1}$} and
$$
  \cM(\omega,\nu) \,:=\, \frac{1}{\sqrt{2\pi\tsigma}} \,\exp\left(-\frac{(\omega-\nu)^2}{2\,\tsigma}\right)\,,
$$
hence one can check that
\begin{align} \label{fInt}
  f_\infty(\omega, \nu) :=  N_\infty(\nu) \,\cM(\omega,\nu)
\end{align}
becomes phase-homogeneous stationary solution to \eqref{K3-1}-\eqref{K3-2}. 

Recently,  the asymptotic stability of $f_\infty$  has been studied in
\cite{CHXZ}, where it has been shown that the  solution  $f$  of \eqref{K3-1}-\eqref{K3-2} converges to
$f_\infty$ exponentially fast for sufficiently large $\tsigma > 0$. More precisely, suppose there exists
$h$ such that $f \,=\, f_\infty \,+\,\sqrt f_\infty\,  h $ and 
\begin{align*}
\int_{\R} \, \left\| h(t, \nu) \right\|_{H^s(\T\times\R)}^2
  \,\dD\nu \,<\, \infty, \quad \forall~t\geq 0,
\end{align*}
for $s\geq 1$, hence  for $\tsigma>0$ satisfying
\begin{align*}
\max\left( \frac{1}{m^2}, \,\tkappa, \,m^2\tkappa^2 \right) \,\ll \, \tsigma,
\end{align*}
there exists $C_1>0$ such that
\begin{align*}
\int_{\R} \, \left\| h(t, \nu) \right\|_{H^s}^2 \,\dD\nu \,\lesssim\, e^{-C_1 t} \, \int_{\R} \, \left\| h(0, \nu) \right\|_{H^s}^2 \,\dD\nu \,.
\end{align*}

Here our aim is twofold. On the one hand,  we
give an explicit condition on the intensity of collision $\tsigma$ and
the coupling strength $\tkappa$ to get the convergence
to the homogeneous stationary state $f_\infty$. This condition requires that $\tkappa$ is sufficiently small compared to $\tsigma$.  On the other hand,
we provide a simpler proof  of
convergence  than the one presented in
\cite{CHXZ}. We apply the  hypocoercivity method with a micro-macro
decomposition,  developed in \cite{Villani:AMS, herau2017, DMS}, to get quantitative estimates on the convergence
to equilibrium. The advantage of this approach is that it is simply
based on a natural weighted $L^2$ estimate. The key tool of our method is a modified energy functional $\cE[f]$, whose square
root is a norm equivalent to the weighted $ L^2$ norm, such that
$$
\frac{\dD \cE[f] }{\dD t} \,\leq\, - C \,\,\cE[f],
$$
for an explicitly computable positive constant $C$. It is worth to mention
that this functional framework is well adapted  to  the development of
structure preserving numerical schemes \cite{BF22,BF23} and will be
the purpose of a forthcoming work \cite{FK:23}. Furthermore, this
technique has also been applied for the asymptotic stability of related
models \cite{helge1,helge2, poyato}.

Before we present our first main result, we introduce macroscopic quantities $N$, $J$, and $P$ defined as
\begin{equation}
  \left\{
  \begin{array}{l}
    \ds N(t,\theta,\nu) \,:=\, \int_{\R} f(t,\bz)\,\dD \omega \quad \left(\mbox{density}\right),
    \\[0.9em]
    \ds J(t,\theta,\nu) \,:=\, \int_{\R} f(t,\bz)\,\left(\omega-\nu\right)\dD \omega \quad \left(\mbox{first moment}\right),
    \\[0.9em]
    \ds P(t,\theta,\nu) \,:=\, \int_{\R} f(t,\bz)\,\left(\omega-\nu\right)^2\dD \omega \quad \left(\mbox{second moment}\right),
  \end{array}
  \right.
  \label{def:N}
\end{equation}
which will be used throughout this paper. One can multiply the equation \eqref{K3-1}-\eqref{K3-2} by $(1,\omega-\nu)$ and integrating in $\omega\in\R$ to get following system of balance laws:
\begin{equation}
  \label{eq:momentsCoupled}
  \left \{ \begin{aligned}
     & \partial_t N \,+\,\partial_\theta \left(J+\nu N\right) \,=\, 0,                                                           \\
     & \partial_t J \,+\,  \partial_\theta\left(P + \nu J\right) \,+\, \tkappa\,(\sin\ast\rho)\, N \,=\, - \frac{J}{m}.
  \end{aligned} \right.
\end{equation}

Next, we introduce  a weighted Lebesgue space  by considering a weight function $\bar\gamma:\R\mapsto \R^+$ such that
$\mybw(\nu)>0$, for all $\nu\in\R$,
\begin{equation}
  \label{hyp:H1}
  \int_{\R} \frac{\dD\nu}{\mybw (\nu)} \,=\, 1 \qquad{\rm and}\qquad
  \int_{\R} |g(\nu)|^2 \,\mybw(\nu)\,\dD\nu \,<\, \infty.
\end{equation}
{Note that $\bar\gamma$ is motivated by the function
  $g$. Indeed, a natural weight would be the function $(\omega,\nu)\mapsto f_\infty^{-1}(\omega,\nu):=\left(
    N_\infty\,\cM(\omega,\nu)\right)^{-1}$, but when  $g$ is compactly
  supported,   $1/g$ is not defined. Therefore,  we introduce $\bar\gamma$ satisfying
  \eqref{hyp:H1}} and  set
$$
  \myw(\omega,\nu) \,:=\, \mybw(\nu)\,\cM^{-1}(\omega,\nu),
$$
{which is motivated by $f_\infty$},  then we  define the weighted $L_\gamma^2(\T\times\R^2)$ norm as
$$
  \| h \|_{L^2_{\myw}} \,=\,  \left(
  \int_{\T\times\R^2} | h|^2 \,\myw(\omega,\nu)\, \dD\bz\right)^{1/2},
$$
and the corresponding weighted $L^2_{\mybw}$ norm for the macroscopic
quantity $R:\T\times\R \mapsto \R$ as
$$
  \|R\|_{L^2_\mybw} \,=\, \left(
  \int_{\T\times\R} | R(\theta,\nu)|^2 \, \mybw(\nu) \, \dD\nu\,\dD\theta\right)^{1/2}\,.
$$
Under this setting, we have exponential relaxation of the solution of \eqref{K3-1}-\eqref{K3-2} toward a phase-homogeneous stationary state.
\begin{theorem}
  \label{th:1}
  Consider an initial data $f_{\rm in}\geq 0$ such that
  $$
\iint_{\T\times\R^2} |f_{\rm in}|^2 \gamma(\omega,\nu)\,\dD \bz \,<\, +\infty.
  $$
 {Then there exists a constant  $C_\infty>0$,} only depending on $\|g\|_{L^2_\mybw}$, such that if
  the coupling strength $\tkappa>0$ and the noise intensity $\tsigma>0$ satisfy
  \begin{align}
    \label{B12}
C_\infty \,\max\left( \sqrt{\frac{\tkappa}{m}}\,, \,\tkappa\,, \,m\tkappa\,, \tkappa^2\, \right) \,\leq\, \tsigma,
  \end{align}
  hence {the solution $f$ to \eqref{K3-1}-\eqref{K3-2} converges to the phase-homogeneous stationary state \eqref{fInt} denoted by $f_\infty$} exponentially fast
  $$
    \|f(t)-f_\infty\|_{L^2_\myw} \,\leq\,3\,\|f_{\rm in}-f_\infty\|_{L^2_\myw}
    \, e^{-C\,t}, \quad \forall~t\geq0,
  $$
 {where $C>0$ only depends on $\tsigma$ and $m$.}
\end{theorem}

The proof of this result is provided in Section \ref{sec:Stab}. The
key idea is to get advantage of the dissipation corresponding to the
Fokker-Planck operator for the weighted $L^2_\gamma$ norm. Then, the
main difficulty consists in proving the convergence of the
macroscopic quantity
$$
N(t,\theta,\nu) \,:=\, \int_\R f(t,\bz) \,\dD \omega,  
$$
toward the equilibrium $N_{\infty}$. We adapt the
hypocoercivity method developed in \cite{Villani:AMS,DMS,lmr:18} to the
present model \eqref{K3-1}-\eqref{K3-2}. {Our objective is to analyze this nonlinear model without imposing
stringent regularity conditions, focusing instead on weighted
$L^2_\gamma$ convergence. It is worth noting that the
hypocoercive  method presented in \cite{Villani:AMS}  is primarily
designed for linear equations and necessitates high
regularity. Similarly, the approach in \cite{DMS} reduces the
regularity requirements but is applied to a large class of linear kinetic
equations and to the one-dimensional Vlasov-Poisson-Fokker-Planck
system supplemented with a compactness argument, as a result, it does
not provide error estimates. Although the technique in \cite{lmr:18}  addresses
nonlinear equations, it also demands higher regularity. To overcome
these limitations,  we combine the strategies from  \cite{DMS} and \cite{lmr:18}, enabling us to handle our nonlinear model effectively. This is possible because the convolution term involving the sine function ensures sufficient regularity, despite the nonlinear nature of our model \eqref{K3-1}-\eqref{K3-2}.}
In detail, instead of estimating directly the
quantities of interest, we introduce modified energy functionals  in order to recover dissipation
and thus a convergence rate on $N$. Our approach is related to the one
developed in \cite{DMS} and \cite{soler, lmr:18, ab:23,addala}   for the Vlasov-Poisson-Fokker-Planck
system. Even if the natural energy corresponding to the system
\eqref{K3-1}-\eqref{K3-2}  does not provide an estimate, the key point here
is to exploit the
regularity of the nonlocal term $\sin\ast\rho$.

Let us now make some comments how our results compares with the one
presented in \cite{CHXZ}.

\begin{remark}
  In \cite{CHXZ}, it is required that $\tkappa>0$ and $\tsigma>0$ satisfy
  \begin{align} \label{B11}
    \tilde C \max\left( \frac{1}{m^2}, \,\tkappa, \,m^2\tkappa^2 \right) \,\leq \, \tsigma,
  \end{align}
  for some sufficiently large $\tilde C>0$. Note that as $\tilde\kappa>0$
  goes to zero, left hand side of \eqref{B11} does not converges to
  zero, whereas left hand side of \eqref{B12} converges to
  zero. Moreover, as $\tkappa>0$ goes to infinity, left hand side of
  \eqref{B11} diverges to infinity with growth rate $\mathcal
  O(m^2\tkappa^2)$, whereas left hand side of \eqref{B12} diverges
  to infinity with growth rate $\mathcal O(\tkappa^2)$.
\end{remark}

{\begin{remark}
For identical oscillators, i.e., $g(\nu) = \delta_0(\nu)$, we can
obtain a similar result. In detail, consider an initial data $f_{\rm in} = f_{\rm in}(\theta, \omega)\geq 0$ such that
$$
\iint_{\T\times\R} |f_{\rm in}(\theta, \omega)|^2\, \cM_0^{-1}(\omega)\,\dD \omega\dD\theta \,<\, +\infty, \quad \mbox{where} \quad \cM_0(\omega) \,=\, \frac{1}{\sqrt{2\pi\,\tsigma}} \, \exp\left(-\frac{\omega^2}{2\tsigma}\right).
$$
Then there exists a universal constant $C_\infty>0$ such that if the coupling strength $\tkappa>0$ and the noise intensity $\tsigma>0$ satisfy
\begin{align}
  \label{B12_bis}
	C_\infty \,\max\left( \sqrt{\frac{\tkappa}{m}}\,, \,\tkappa\,, \,m\tkappa\,, \tkappa^2\, \right) \,\leq\, \tsigma,
\end{align}
hence the solution $f = f(t, \theta, \omega)$ to
\begin{align} \label{sys:idKS}
	\begin{dcases}
		\partial_t f \,+\,\omega \,\partial_\theta f \,-\,\tkappa\,
		(\sin\ast\rho) \,\partial_\omega f \,=\, \frac{1}{m} \,\partial_\omega \bigg( \tsigma\,\partial_\omega f \,+\,\omega\,f \bigg)\,,
		\\[0.5em]
		\rho(t, \theta) \,=\, \int_{\mathbb R} f(t, \theta, \omega) \, \dD\omega\,, \\[.9em]
		f(t=0) \,=\, f_{\rm in} \geq 0\,.
	\end{dcases}
\end{align}
converges to the phase-homogeneous stationary state $\cM_0/(2\pi)$ exponentially fast
$$
\left\|f(t)-\frac{\cM_0}{2\pi} \right\|_{L^2_{\cM_0^{-1}}} \leq\,\,\,\,\,3\,\left\|f_{\mathrm{in}}-\frac{\cM_0}{2\pi} \right\|_{L^2_{\cM_0^{-1}}}
\, e^{-C\,t}, \quad \forall~t\geq0,
$$
where $C>0$ only depends on $\tsigma$ and $m$. We omit the proof for this case since the argument is identical.

On the other hand, a natural question arises as to whether there
exists a universal framework in which both Theorem \ref{th:1} and the
case $g(\nu) = \delta_0(\nu)$ can be addressed. One possible candidate
is to regard $f(t, \theta, \omega)$ as a Radon probability measure on
$\mathbb R$ for fixed $(\theta, \omega)\in\mathbb T\times\mathbb
R$.
However, considering the variables separately in this manner
prevents  energy  estimates  from being carried out as they are 
stronly based on a $L^2$ framework.
\end{remark}
}


\subsection{Diffusion limit  for identical oscillators}
We now present our second main result on the  diffusion limit of \eqref{K3-1}-\eqref{K3-2}  for
identical oscillators, {\it i.e.}, $g(\nu) = \delta_0(\nu)$. {In this case, particle model \eqref{sys:Ku} reduces to
\begin{align*}
	\begin{dcases}
		\dD\theta^i_t \,=\, \omega^i_t \, \dD t, \\[.9em]
		m\, \dD\omega^i_t \,=\, \left( -\omega^i_t \,+\, \frac{\kappa}{N} \sum_{k=1}^N \sin\left( \theta^k_t-\theta^i_t \right) \right) \dD t \,+\, \sqrt{2\sigma} \, \dD B^i_t,
	\end{dcases}
\end{align*}
and the corresponding Vlasov-Fokker-Planck-type equation becomes \eqref{sys:idKS}. We remark that the solution $f = f(t, \theta, \omega)$ no longer depends on $\nu$, and it can be interpreted as the conditional probability of the solution to \eqref{K3-1}-\eqref{K3-2} under $g(\nu) = \delta_0(\nu)$.}

Now we consider the following rescaling $t \mapsto \eps\,t$ and $\eps=m$
in \eqref{sys:idKS}, which yields :
{\begin{align} \label{K4-1}
  \begin{dcases}
  \eps\,\partial_t f^\eps \,+\,\omega \,\partial_\theta f^\eps \,-\,\tkappa\,
  (\sin\ast\rho^\eps) \,\partial_\omega f^\eps \,=\, \frac{1}{\eps} \, \cL_{\mathrm{FP}}[f^\eps]\,,
    \\[0.9em]
    f^\eps(t=0) \,=\, f_{\rm in}^\eps \geq 0\,,
  \end{dcases}
\end{align}
with
\begin{equation}
  \label{K4-2}
  \cL_{\mathrm{FP}}[f^\eps] \,:=\, \partial_\omega \bigg(
  \tsigma\,\partial_\omega f^\eps \,+\,\omega f^\eps \bigg)
\end{equation}}
and the density $\rho^\eps$ is given by
$$
\rho^\eps(t,\theta) \,=\, \int_\R f^\eps(t,\bz)
\,\dD\omega,\qquad \forall \, (t,\theta)\,\in\,\R^+\times\T,
$$
{where $\bz = (\theta, \omega)\in\mathbb T\times\mathbb R$. Then we again introduce macroscopic quantities $J^\varepsilon$ and $P^\varepsilon$ defined as
\begin{equation}
	\left\{
	\begin{array}{l}
		\ds J^\varepsilon(t,\theta) \,:=\, \int_{\R} \omega\,f^\eps(t,\bz)\,\dD \omega \quad \left(\mbox{first moment}\right),
		\\[0.9em]
		\ds P^\varepsilon(t,\theta) \,:=\, \int_{\R} \omega^2\,f^\eps(t,\bz)\,\dD \omega \quad \left(\mbox{second moment}\right).
	\end{array}
	\right.
	\label{macro_eps}
\end{equation}
Note that $N^\eps$ need not be defined, as it is identical to $\rho^\eps$.}
We remind \eqref{eq:momentsCoupled} that $\rho^\eps$ and the first moment $J^\eps$ satisfy
\begin{align}
  \left \{ \begin{aligned} \label{eq:macro}
     & \partial_t \rho^\eps \,+\,\frac{1}{\eps}\partial_\theta J^\eps \,=\, 0,                                                           \\
     & \eps\,\partial_t J^\eps \,+\,  \partial_\theta P^\eps  \,+\, \tkappa\,(\sin\ast\rho^\eps)\, \rho^\eps \,=\, - \frac{J^\eps}{\eps}\,,
  \end{aligned} \right.
\end{align}
hence differentiating the last equation with respect to $\theta$ and
combining it with the  the former, it yields that 
\begin{equation}
  \label{eq:moments}
\partial_t \left(\rho^\eps -\eps \partial_\theta J^\eps\right)  -
  \partial_\theta\biggl(  \partial_\theta P^\eps  \,+\,
    \tkappa\,(\sin\ast\rho^\eps)\, \rho^\eps \biggr) \,=\, 0. 
\end{equation}
In the limit $\eps\rightarrow 0$, it is expected that
$(f^\eps)_{\eps>0}$ converges to $\rho\,\cM$, {where $\cM$ is now taken to be the centered Gaussian distribution
$$
\cM(\omega) \,=\, \frac{1}{\sqrt{2\pi\,\tsigma}} \, \exp\left(-\frac{\omega^2}{2\tsigma}\right).
$$
Then, we have} 
$$
P^\eps = \int_{\R} f^\eps \omega^2 \dD\omega \,\rightarrow \;
\tsigma\, \rho, \quad{\rm as}\,\,\eps\rightarrow 0. 
$$
Therefore, we formally get that the limit $\rho$ is solution  to the following drift-diffusion equation
\begin{equation}
  \label{eq:dd}
  \left\{
  \begin{array}{l}
  \ds\partial_t \rho - \tsigma\,\partial_{\theta}^2\rho  \,-\, \tkappa \,\partial_\theta\left( (\sin\ast\rho) \,\rho\right)
  \,=\, 0\,,
  \\[0.9em]
  \ds\rho(t=0) \,=\, \rho_{\rm in}.
  \end{array}\right.
\end{equation}
Since the $\nu$ variable now cancels, {we consider the weighted $L^2_{\cM^{-1}}$ space where the weight $\cM^{-1}$ is now given by the inverse of the centered Gaussian distribution $\cM$.}

Thus, we prove the following result.

\begin{theorem}
  \label{th:2}
  Suppose that the initial data $(f^\eps_{\rm in})_{\eps>0}$ in
  \eqref{K4-1}-\eqref{K4-2}, satisfy the following assumptions
  \begin{equation}
    \label{hyp:th:2}
{\sup_{\varepsilon>0} \left( \|f^\eps_{\rm in}\|_{L^2_{\cM^{-1}}} \,+\,   \|\partial_\theta
   f^\eps_{\rm in}\|_{L^2_{\cM^{-1}}} \right) \,<\, +\infty,}
   \end{equation}
	and the initial datum $\rho_{\rm
     in}$ in \eqref{eq:dd} verifies
   $$
\|\rho_{\rm in} \|_{L^2} \,<\, +\infty.
   $$
   Moreover, we suppose that
   \begin{equation*}
{\|f^\eps_{\rm in} \|_{L^1} \,=\,\| \rho_{\rm in}\|_{L^1} \,=\, 1.}
\end{equation*}
Let $f^\eps$ be the solution to \eqref{K4-1}-\eqref{K4-2}  and $\rho$
be the solution to \eqref{eq:dd}.  Then  the following statements hold true uniformly with respect to $\eps$
\begin{equation*}
    \ds\|f^\eps(t)-\rho^\eps(t)\,\cM\|_{L^2_{\cM^{-1}}} \,\leq\,
    \|f^\eps_{\rm in}-\rho^\eps_{\rm in}\,\cM\|_{L^2_{\cM^{-1}}}
    e^{-\tsigma\,t/(4\eps^2)} \,+\,
                                                         C\,\eps\,  \left(  \|\partial_\theta
   f^\eps_{\rm in}\|_{L^2_{\cM^{-1}}} +  \|
   f^\eps_{\rm in}\|_{L^2_{\cM^{-1}}}  \right)\,e^{C\,t}\,,
\end{equation*}
and
                                                         \begin{equation}
                                                           \label{cv:rho}
\left\| \rho^\eps(t)-\rho(t)\right\|_{H^{-1}}\,\leq\, C\left(
\left\|\rho^\eps_{\rm in} -\rho_{\rm in}\right\|_{H^{-1}} \,+\, \eps
\left(\|f^\eps_{\rm in}\|_{L^2_{\cM^{-1}}}
  + \|\partial_\theta f^\eps_{\rm in}\|_{L^2_{\cM^{-1}}}\right)
\right) \, \,e^{C\,t}\,, 
\end{equation}
where  $C$ is a positive constant  only depending on $\tsigma$ and $\tkappa$.
\end{theorem}

\begin{remark}
In \cite{HSZ}, the authors also investigated the same problem using a
compactness method and an energy estimate, they proved that 
$$
\rho^\eps \,\rightarrow \rho,  \quad {\rm in}\,\,
L^1\left((0,T)\times\T\right), \quad{\rm when}\,\, \eps\,\rightarrow\, 0\,,
$$
considering the notion of  renormalized solution.  
{Our method is simpler than the previous approach and additionally
enables us to establish error estimates with respect to $\eps$, as shown in  \eqref{cv:rho}.
However, it requires higher regularity in  $\theta$, uniformly with
respect to $\eps$,
\begin{equation*}
\|f^\eps(t)\|_{L^2_{\cM^{-1}}} \,+\,   \|\partial_\theta
f^\eps(t)\|_{L^2_{\cM^{-1}}}  \,\leq\, \left(  \|\partial_\theta
f^\eps_{\rm in}\|_{L^2_{\cM^{-1}}} \,+\, 3\, \|
f^\eps_{\rm in}\|_{L^2_{\cM^{-1}}}  \right) \, e^{\tkappa^2 \,t \,/\, \tsigma}\,,
\end{equation*}
which will be shown in Proposition \ref{prop:L2}.} \\
 Recently, A. Blaustein  provided  error estimates for the
  diffusive limit of the  Vlasov-Poisson-Fokker-Planck system by
  proving propagation of  regularity in weighted $L^p$ spaces \cite{ab:23}. 
\end{remark}

\begin{remark}
 {Let us emphasize that the assumption of identical
    oscillators is crucial to perform the diffusive limit.}
{Consider the same rescaling $t \mapsto \eps\,t$ and $\eps=m$ in \eqref{K3-1}-\eqref{K3-2},  then \eqref{eq:momentsCoupled} becomes
\begin{equation*}
	\left \{ \begin{aligned}
		& \partial_t N^\eps \,+\,\frac{1}{\eps}\partial_\theta \left(J^\eps+\nu N^\eps\right) \,=\, 0,                                                           \\
		& \eps\partial_t J^\eps \,+\,  \partial_\theta\left(P^\eps + \nu J^\eps\right) \,+\, \tkappa\,(\sin\ast\rho^\eps)\, N^\eps \,=\, - \frac{J}{\eps},
	\end{aligned} \right.
\end{equation*}
where macroscopic quantities are given by
\begin{equation*}
	\left\{
	\begin{array}{l}
		\ds N^\eps(t,\theta,\nu) \,:=\, \int_{\R} f^\eps(t,\bz)\,\dD \omega \quad \left(\mbox{density}\right),
		\\[0.9em]
		\ds J^\eps(t,\theta,\nu) \,:=\, \int_{\R} f^\eps(t,\bz)\,\left(\omega-\nu\right)\dD \omega \quad \left(\mbox{first moment}\right),
		\\[0.9em]
		\ds P^\eps(t,\theta,\nu) \,:=\, \int_{\R} f^\eps(t,\bz)\,\left(\omega-\nu\right)^2\dD \omega \quad \left(\mbox{second moment}\right).
	\end{array}
	\right.
\end{equation*}
This yields
\begin{align*}
\partial_t \left(N^\eps-\eps\partial_\theta J^\eps \right) \,+\, \frac{\nu}{\eps}\,\partial_\theta N^\eps \,-\,\partial_\theta \biggl( \partial_\theta\left(P^\eps + \nu J^\eps\right) \,+\, \tkappa\,(\sin\ast\rho^\eps)\, N^\eps \biggr) \,=\,0\,,
\end{align*}
whose second term blows up as $\eps\rightarrow 0$. Therefore, we only consider $g(\nu) = \delta_0(\nu)$ in the diffusive limit.}
\end{remark}

The rest of the paper is organized as follows.  On the one hand, in the next section (Section \ref{sec:Stab}), we establish some basic properties of
\eqref{K3-1}-\eqref{K3-2}, hence we study the propagation of the
modified energy functional $\mathcal E[f]$ and prove Theorem \ref{th:1} on the
exponential relaxation of $f$ toward the phase homogeneous stationary
state  $f_\infty$. On the other hand, in Section \ref{sec:diff}, we consider the
particular case when all oscillators are identical and study the
diffusion limit  to prove our second result (Theorem \ref{th:2}) on
error estimates in the diffusion limit for identical oscillators. Finally, Section \ref{sec:Con} is devoted to a brief summary and possible future works.

%
%
%
%
%
%
%

\section{Long time behavior for small coupling strength}
\label{sec:Stab}
\setcounter{equation}{0}

In this section, {we outline the proof of Theorem \ref{th:1}. Then,} we present {\it a priori} estimates which aim at describing
the   long time
behavior of \eqref{K3-1}-\eqref{K3-2} when collisions dominate. Finally, we focus
on macroscopic quantities and provide a free energy estimate, which is
the starting point of our analysis.

{\subsection{Outline of the proof of Theorem \ref{th:1}} \label{sec:outline1}
We first aim to study the propagation of the weighted $L^2_\gamma$
norm.  Specifically, when the parameter $\tilde\sigma/m$ is large
enough, we establish that  (Proposition \ref{prop:Diss}),
\begin{align*}
	\frac12 \frac{\dD}{\dD t}
  \|f(t)-f_\infty\|_{L^2_{\myw}}^2\,\lesssim \, - \cI[f](t)  \,+\,  \| N(t)-N_\infty\|_{L^2_\mybw}^2
\end{align*}
where $\cI[f](t)\geq 0$ is  the dissipation  associated with  the
Fokker-Planck operator for the weighted $L^2_\gamma$
norm. Unfortunately the dissipation $\cI[f]$ does not directly control the
functional  $\|f(t)-f_\infty\|_{L^2_{\myw}}^2$. However, as shown in
Lemma \ref{lem:GPoincare}, we have
\begin{align}\label{ineq:fIN}
	\|f(t)-f_\infty\|_{L^2_\myw}^2  \,\leq\, \cI[f](t) \,+\,\| N(t)-N_\infty\|_{L^2_\mybw}^2\,.
\end{align}
The goal, therefore, is to modify the weighted $L^2_\myw$ norm to
obtain a new dissipation estimate that better controls
$\|N(t)-N_\infty\|_{L^2_\mybw}^2$.
\\
A key observation  is that as $f(t)$ approaches $f_\infty$, the macroscopic
quantity $P$,  defined in \eqref{eq:momentsCoupled},  satisfies
$P-\tilde\sigma N$ converges to zero. Consequently,   the equation
\eqref{eq:momentsCoupled}  governing the
momentum $J$ can be rewritten to reflect this asymptotic behavior as
$$
\partial_t J \,+\,  \tilde\sigma\partial_\theta (N-N_\infty) \,+\, \partial_\theta\left(P-\tilde\sigma N + \nu J\right) \,+\, \tkappa\,(\sin\ast\rho)\, N \,=\, - \frac{J}{m}.
$$
To this end, we select  $\cA(t)$ such that:
\begin{equation}
  \label{def:A1}
 \cA(t) \,:=\,  -\int_{\T\times \R}  \partial_\theta J(t)\,v(t)
 \,\mybw\,\dD\nu \dD \theta \,=\; \int_{\T\times \R}  J(t)\,\partial_\theta v(t)
 \,\mybw\,\dD\nu \dD \theta,
 \end{equation}
 where $v(t)$ satisfies:
 $$
 -\partial_{\theta}^2 v \,=\, N-N_\infty.
 $$
 We then define the modified energy functional $\cE[f]$ as
 $$
\cE[f](t) \,:=\, \frac{1}{2}\,\|f(t)-f_\infty\|_{L^2_{\myw}}^2 \,+\, \alpha \,   \cA(t)\,,
$$
where $\alpha$ is a free parameter, chosen so that $\cE[f]$
is equivalent to  $\|f(t)-f_\infty\|_{L^2_{\myw}}^2$. Finally,  with an appropriate choice of $\alpha$, we aim to establish:
$$
\frac{\dD\cE[f] }{\dD t}(t) \,\lesssim\,  - \cE[f],
$$
which implies exponential decay of the modified energy and,
consequently, of $f(t)- f_\infty$.}

\subsection{Basic properties}
In this section, we study some basic properties of the inertial
equation \eqref{K3-1}-\eqref{K3-2} showing the propagation of the weighted
$L^2_\gamma$ norm and estimate some macroscopic quantities for latter
use. First we remind the estimate provided in \cite{CHXZ, HSZ}.
\begin{proposition}
  \label{prop:basic}
  Let $f = f(t, \theta, \omega, \nu)$ be a classical solution to
  \eqref{K3-1}-\eqref{K3-2} with a nonnegative initial data $f_{\rm in}\in
  L^1(\T\times\R^2)$. Then, for all time $t\geq 0$, we have that
  $f(t)$ is also nonnegative and
  \begin{align*}
     & \int_{\T\times\R^2} f(t, \bz) \,\dD\bz \,=\,
       \int_{\T\times\R^2} f_{\rm in}(\bz) \,\dD\bz,       \\[0.9em]
     & \int_{\T\times\R} f(t, \bz) \,\dD\omega\,\dD\theta \,=\,
       \int_{\T\times\R} f_{\rm in}(\bz) \,\dD\omega\,\dD\theta \,=\, g(\nu).
  \end{align*}
\end{proposition}

%
%

We aim to study the propagation of the weighted $L^2_\myw$ norm and
first prove
the following preliminary result.
\begin{lemma} \label{lem:GPoincare}
  Let $f = f(t, \theta, \omega, \nu)$ be a classical solution to \eqref{K3-1}-\eqref{K3-2}. Then for all time $t\geq 0$, we have
  \begin{align*}
    \|f(t)-f_\infty\|_{L^2_\myw}^2  \,\leq\, \cI[f](t) \,+\,\| N(t)-N_\infty\|_{L^2_\mybw}^2,
  \end{align*}
 where $\cI[f](t)$ corresponds to the dissipation of the Fokker-Planck
 operator and  is defined as
  \begin{align}
    \label{def:I}
    \cI[f](t) \,:=\, \int_{\T\times \R^2} \left| \partial_\omega
    \left( \frac{f(t)}{\cM} \right) \right|^2 \,\mybw(\nu)\cM(\omega,\nu)\,\dD \bz \,\geq\, 0.
  \end{align}
\end{lemma}
\begin{proof}
  It follows
  \begin{equation}
    \label{equ:fInf}
  \begin{array}{lll}
  	\|f-f_\infty\|_{L^2_{\myw}}^2  &=& \|f-
  	N\,\cM\|_{L^2_{\myw}}^2 \,+\,\| (N-N_\infty)\, \cM
  	\|_{L^2_{\myw}}^2,\\
  	&=&   \|f-
  	N\,\cM\|_{L^2_{\myw}}^2 \,+\,\| N-N_\infty\|_{L^2_\mybw}^2.
  \end{array}
  \end{equation}
  We use the Gaussian-Poincar\'e inequality with respect to probability measure
  $\cM \dD \omega$ to obtain
  \begin{eqnarray*}
    \|f- N \cM\|_{L^2_{\myw}}^2 &=&\int_{\T\times\R} \left(\int_\R |f-N\,\cM|^2\cM^{-1}(\omega,\nu)\,\dD\omega\right)\,\mybw(\nu) \,\dD\nu\dD\theta \\
    &\leq&\, \int_{\T\times \R}\left(\int_\R \left|\partial_\omega\left(\frac{f}{\cM}\right)\right|^2\cM(\omega,\nu)\,\dD\omega\right)\,\mybw(\nu) \,\dD\nu\dD\theta\,,
  \end{eqnarray*}
  hence we have
  $$
    \|f(t)-f_\infty\|_{L^2_{\myw}}^2  \,\leq\, \cI[f](t) \,+\,\| N(t)-N_\infty\|_{L^2_\mybw}^2.
  $$
\end{proof}


From this latter lemma, we can prove the key estimate on the dissipation of the weighted $L^2$ norm.
\begin{proposition}
  \label{prop:Diss}
  Let $f \,=\, f(t, \theta, \omega, \nu)$ be a classical solution to \eqref{K3-1}-\eqref{K3-2}. Then for all time $t\geq 0$, we have
  \begin{eqnarray*}
    \frac12 \frac{\dD}{\dD t} \|f(t)-f_\infty\|_{L^2_{\myw}}^2  &\leq&
    -\left[\frac{\tsigma}{m}  \,-\, \frac{\tkappa}{2}
      \left(3+\frac{1}{2\sqrt\pi}\,\|g\|_{L^2_\mybw}\right)\right] \,\cI[f](t)
    \\[0.9em]
    &&
    +\,\frac{\tkappa}{2} \left(1+\frac{1}{2\sqrt\pi}
    \,\|g\|_{L^2_\mybw}\right)\,\|N(t)-N_\infty\|_{L^2_\mybw}^2\,,
  \end{eqnarray*}
where the dissipation $\cI[f](t)$ is defined  in \eqref{def:I}.
\end{proposition}


\begin{proof}
{We use Proposition \ref{prop:basic} and the definition of the phase-homogeneous state $f_\infty = N_\infty\cM$ and the weight $\myw=\cM^{-1}\,\mybw$ to have
\begin{eqnarray*}
\frac12 \frac{\dD}{\dD t} \int_{\T\times \R^2} |f-f_\infty|^2 \,\myw \,\dD \bz
&=& \int_{\T\times \R^2} \partial_tf \,(f-f_\infty) \,\myw \,\dD \bz \\
&=& \int_{\T\times \R^2} f \, \myw \, \partial_tf \, \dD \bz \,-\, \int_{\T\times \R^2} \left( N_\infty \, \mybw \right) \, \partial_tf \,\dD \bz \\
&=& \int_{\T\times \R^2} f \, \myw \, \partial_tf \, \dD \bz \,-\, \frac{1}{2\pi} \frac{\dD}{\dD t} \int_{\R} |g|^2 \, \mybw \,\dD \nu \,=\, \int_{\T\times \R^2} f \, \myw \, \partial_tf \, \dD \bz \,.
\end{eqnarray*}}
It follows
  \begin{eqnarray*} 
    \frac12 \frac{\dD}{\dD t} \|f-f_\infty\|_{L^2_{\myw}}^2
    &=& \int_{\T\times \R^2} \frac{f}{\cM} \left( -\omega
    \partial_\theta f \,+\,\tkappa\, (\sin\ast\rho)\partial_\omega f
    +\mathcal L_{\mathrm{FP}}[f] \right) \,\mybw\,\dD \bz \\
    \\
    &=& -\int_{\T\times \R^2} \partial_\theta \left( \frac{\omega
      f^2}{2\,\cM} \right) \,\mybw\,\dD \bz \,+\,\tkappa\, \int_{\T\times
      \R^2} \frac{f}{\cM} \, (\sin\ast\rho) \, \partial_\omega
    f \,\mybw\,\dD \bz
    \\
    && +\int_{\T\times \R^2} \frac{f}{\cM}\, \cL_{\mathrm{FP}}[f] \,\mybw\,\dD \bz
    \\
    & =& \int_{\T\times \R^2} \frac{f}{\cM}\, \cL_{\mathrm{FP}}[f] \,\mybw\,\dD \bz \,-\, \tkappa\,\int_{\T\times \R^2} (\sin\ast\rho) \, f\,
    \partial_\omega \left( \frac{f}{\cM} \right)\, \mybw\,\dD \bz
  \end{eqnarray*}
  We substitute
  \begin{align*}
    \partial_\omega \cM \,=\, -\frac{\omega-\nu}{\tsigma} \,\cM
  \end{align*}
  into $\cL_{\mathrm{FP}}$ to observe
  \begin{eqnarray*}
    \int_{\T\times \R^2} \frac{f}{\cM}\, \cL_{\mathrm{FP}}[f] \,\mybw\,\dD \bz &=& \frac{\tsigma}{m} \,\int_{\T\times \R^2} \frac{f}{\cM}\, \partial_\omega \left( \partial_\omega f -\frac{\partial_\omega \cM}{\cM}\, f \right) \,\mybw\, \dD\bz \\
    &=& -\frac{\tsigma}{m} \,\int_{\T\times \R^2} \left(\partial_\omega \left( \frac{f}{\cM} \,\cM \right) -\frac{f}{\cM} \partial_\omega \cM \right) \,\partial_\omega \left( \frac{f}{\cM} \right) \,\mybw\,\dD\bz
    \\
    &=& -\frac{\tsigma}{m} \,\int_{\T\times \R^2} \left| \partial_\omega \left( \frac{f}{\cM} \right) \right|^2 \,\cM\,\mybw\,\dD\bz \,=\, -\frac{\tsigma}{m} \,\cI[f]\,.
  \end{eqnarray*}
  We combine the latter results to obtain
  \begin{align*}
    \frac12 \frac{\dD}{\dD t} \|f(t)-f_\infty\|_{L^2_{\myw}}^2 \,=\, -\frac{\tsigma}{m} \,\cI[f](t)\,-\, \tkappa \,\int_{\T\times \R^2} (\sin\ast\rho(t))\, f(t)\,
    \partial_\omega \left( \frac{f(t)}{\cM}\right) \,\mybw\,\dD\bz \,.
  \end{align*}
  It is left to estimate the first term of the right hand side. Using
  that
  $$
    \sin\ast\rho \,=\, \sin\ast\left(\rho-\rho_\infty\right),
  $$
  and since $f_\infty=N_\infty\,\cM$, we have
  \begin{eqnarray*}
    \int_{\T\times \R^2} (\sin\ast\rho)\, f\, \partial_\omega \left( \frac{f}{\cM} \right) \,\mybw\, \dD\bz  & =&  \int_{\T\times \R^2} \sin\ast\rho
    \,\frac{f}{\sqrt{\cM}} \,\sqrt{\cM}\, \partial_\omega
    \left(\frac{f}{\cM} \right) \,\mybw\, \dD \bz
    \\
    & =&  \int_{\T\times \R^2} \sin\ast\rho
    \,\frac{f-f_\infty}{\sqrt{\cM}} \sqrt{\cM} \partial_\omega
    \left(\frac{f}{\cM}
    \right) \,\mybw\,\dD \bz
    \\
    & +&  \int_{\T\times \R^2}  \left( \sin\ast\left(\rho-\rho_\infty\right) \right)
    \,N_\infty\,\sqrt{\cM} \,\sqrt{\cM} \,\partial_\omega
    \left(\frac{f}{\cM}
    \right) \,\mybw\,\dD \bz\,.
  \end{eqnarray*}
  Then we get
  \begin{eqnarray*}
    &&\left| \tkappa\, \iint_{\T\times \R^2} (\sin\ast\rho)\, f\,
    \partial_\omega \left( \frac{f}{\cM} \right) \,\mybw\,\dD\bz \right|
    \\[0.9em]
    &&\leq\, \tkappa \, \left( \|\sin\ast\rho\|_{L^\infty}
    \|f-f_\infty\|_{L^2_{\myw}} \,+\,
    \|\sin\ast(\rho-\rho_\infty)\|_{L^\infty}  \|N_\infty\|_{L^2_\mybw}\right)\, \sqrt{\cI[f]}(t).
  \end{eqnarray*}
  It follows from Young's convolution inequality that
  $$
    \left\{
    \begin{array}{l}
      \ds\|\sin\ast\rho\|_{L^\infty}\,\leq\, \|\rho\|_{L^1} \,=\, \|f\|_{L^1}\,=\,1\,,
      \\[0.9em]
      \ds {\|\sin\ast\left(\rho-\rho_\infty\right) \|_{L^\infty} \,\leq\, \|\sin\|_{L^2} \, \|\rho-\rho_\infty\|_{L^2} \,=\, \sqrt{\pi}\,\|\rho-\rho_\infty\|_{L^2}.}
    \end{array}
    \right.
  $$
  On the other hand, using \eqref{hyp:H1} and 
  \begin{eqnarray*}
    \rho - \rho_\infty &=& \int_{\R} ( N - N_\infty ) \,\dD\nu
    \,=\, \int_{\R}\frac{1}{\mybw^{1/2}}\, ( N - N_\infty ) \,\mybw^{1/2} \,\dD\nu \\
    &\leq& \left( \int_{\R} ( N - N_\infty )^2 \, \mybw \,\dD\nu \right)^{1/2} \,,
  \end{eqnarray*}
we get that
  $$
    \|\rho-\rho_\infty\|_{L^2}\, \leq\,  \|N-N_\infty\|_{L^2_\mybw}\,.
  $$
  Therefore, it yields  the following estimate
  \begin{equation}
    \label{absBound}
  	\begin{array}{lll}
    &&\left| \tkappa\,\iint_{\T\times \R^2} \sin\ast\rho\, f\,
    \partial_\omega \left( \frac{f}{\cM} \right) \,\mybw\,\dD\bz \right|
    \\[0.9em]
    && \leq\,
    \tkappa\, \left(
     \|f-f_\infty\|_{L^2_{\myw}}  +\sqrt{\pi}\, \|N_\infty\|_{L^2_\mybw}\,\|N-N_\infty\|_{L^2_\mybw} \right) \,\sqrt{\cI[f]}(t).
     \end{array}
  \end{equation}
  Now using Lemma \ref{lem:GPoincare}, we get   that
$$
    \left\{
    \begin{array}{l}
      \ds\|f(t)-f_\infty\|_{L^2_\myw} \,\leq\, \sqrt{\cI[f]}(t) \,+\, \|N(t)-N_\infty\|_{L^2_\mybw}\,,
      \\[0.9em]
      \ds  \sqrt{\cI[f]}(t) \,\|N(t)-N_\infty\|_{L^2_\mybw} \,\leq\, \frac{1}{2}\left(\cI[f](t) \,+\, \|N(t)-N_\infty\|_{L^2_\mybw}^2\right),
    \end{array}\right.
  $$
  and gathering the latter results, it gives
  \begin{eqnarray*}
    \frac12 \frac{\dD}{\dD t} \|f(t)-f_\infty\|_{L^2_{\myw}}^2  &\leq&
    -\left[\frac{\tsigma}{m}  \,-\, \frac{\tkappa}{2}
   \,   \left(3+\sqrt{\pi}\,\|N_\infty\|_{L^2_\mybw}\right)\right] \,\cI[f](t)
    \\[0.9em]
    &&
    +\frac{\tkappa}{2}\, \left(1+\sqrt{\pi} \,\|N_\infty\|_{L^2_\mybw}\right)\,\|N(t)-N_\infty\|_{L^2_\mybw}^2\,.
  \end{eqnarray*}
 {Finally, $\|N_\infty\|_{L^2_\mybw} = \|g\|_{L^2_\mybw}/(2\pi)$ yields the desired inequality.}
  \end{proof}


Next, we remind the macroscopic quantities $N$, $J$, and $P$ defined
in \eqref{eq:momentsCoupled}, which satisfy
\begin{equation*}
  \left \{ \begin{aligned}
     & \partial_t N \,+\,\partial_\theta \left(J+\nu N\right) \,=\, 0,                                                           \\
     & \partial_t J \,+\,  \partial_\theta\left(P + \nu J\right) \,+\, \tkappa\,(\sin\ast\rho)\, N \,=\, - \frac{J}{m}.
  \end{aligned} \right.
\end{equation*}
This will be used later to define a modified energy functional
$\cE[f]$ to prove convergence to equilibrium. Before to do that, we prove the following estimates on the macroscopic moments.

\begin{lemma}[Moments estimates]
  Let $J$ and $P$ be the moments given by \eqref{def:N}. Then it holds for all time
  $t\geq 0$ that
  \begin{equation*}
    \left\{
    \begin{array}{l}
      \ds\| J(t)\|_{L^2_\mybw}^2\, \leq\,    \tsigma \, \cI[f](t)\,,
      \\[0.9em]
      \ds \left\|P(t) - \tsigma\, N(t) \right\|_{L^2_\mybw}^2 \,\leq\,
      3\,\tsigma^2  \,\cI[f](t)\,.
    \end{array}\right.
  \end{equation*}
  \label{cs}
\end{lemma}


\begin{proof}
  For the  first estimate, we observe that  $J$  reads as follows
  \begin{eqnarray*}
    J &=& \int_\R  (\omega-\nu) \,\cM^{1/2} \, f
    \,\cM^{-1/2}\,\dD \omega\\
    &=&
    \int_\R (\omega-\nu)\, \cM^{1/2}\,(f - N \, \cM)\, \cM^{-1/2}\,\dD \omega.
  \end{eqnarray*}
  Hence, using that
  $$
    \int_{\R} (\omega-\nu)^2 \cM \,\dD\omega \,=\, \tsigma
  $$
  and applying  the Cauchy-Schwarz inequality, we obtain
\begin{align} \label{inequ:J}
    \|  J(t) \|_{L^2_\mybw}^2 \, \leq\, \, \tsigma \, \|f(t)-
    N(t) \,\cM \|_{L^2_{\myw}}^2,\quad\forall ~t \geq 0\,.
\end{align}
  As in the proof of Lemma \ref{lem:GPoincare}, from the Gaussian Poincaré inequality we have
  $$
    \|  J(t) \|_{L^2_\mybw}^2 \,\leq\, \tsigma\, \cI[f](t),\quad\forall ~t \geq 0\,.
  $$
  Finally, we also have
  \begin{eqnarray*}
    P  \,-\, \tsigma N &=& \int_\R
    \left(\omega-\nu\right)^2 \,f\,\dD \omega \,-\, \int_{\R} (\omega-\nu)^2 \, N \, \cM \,\dD\omega,
    \\
    &=& \int_\R  (\omega-\nu)^2 \cM^{1/2}\,(f\,-\, N\, \cM)\,
    \cM^{-1/2}\,\dD \omega\,.
  \end{eqnarray*}
  We proceed as before using that
  $$
    \int_\R (\omega-\nu)^4 \cM \,\dD\omega \,\,=\,\, 3\,\tsigma^2,
  $$
  hence we get for any $t\geq 0$,
  \begin{eqnarray*}
    \left\|P (t)   \,-\, \tsigma N(t) \right\|_{L^2_{\mybw}}^2  &\leq & 3\,\tsigma^2\, \|f(t)-
    N(t) \,\cM \|_{L^2_{\myw}}^2
    \\
    &\leq &3\,\tsigma^2\, \cI[f](t)\,.
  \end{eqnarray*}
\end{proof}

{\begin{remark} \label{remark_eps1}
For the latter usage in Section \ref{sec:diff}, we note that the above lemma still holds for the case of $g(\nu) = \delta_0(\nu):$
\begin{equation*}
	\left\{
	\begin{array}{l}
		\ds\| J^\eps(t)\|_{L^2}^2\, \leq\,    \tsigma \, \cI[f^\eps](t)\,,
		\\[0.9em]
		\ds \left\|P^\eps(t) - \tsigma\, \rho^\eps(t) \right\|_{L^2}^2 \,\leq\, 3\,\tsigma^2\, \|f^\eps(t)-
		\rho^\eps(t) \,\cM \|_{L^2_{M^{-1}}}^2 \,\leq\,
		3\,\tsigma^2  \,\cI[f^\eps](t)\,.
	\end{array}\right.
\end{equation*}
where $f^\eps = f^\eps(t, \theta, \omega)$ is the solution to \eqref{K4-1}-\eqref{K4-2}, $J^\eps$ and $P^\eps$ are defined in \eqref{macro_eps}, and $\cI[f^\varepsilon](t)$ is defined as
\begin{align*}
	\cI[f^\eps](t) \,:=\, \int_{\T\times \R} \cM(\omega) \, \left| \partial_\omega
	\left( \frac{f^\eps(t)}{\cM} \right) \right|^2 \,\dD \omega \dD\theta \,\geq\, 0.
\end{align*}
In this case, $\cM = \cM(\omega)$ is the centered Gaussian distribution $:$
$$
\cM(\omega) \,=\, \frac{1}{\sqrt{2\pi\,\tsigma}} \, \exp\left(-\frac{\omega^2}{2\tsigma}\right).
$$
We omit the proof since the argument is identical.
\end{remark}
}

Now, we study the asymptotic stability of $f_\infty$ when $\tsigma \gg \tkappa$.  The goal is first
to modify $\|f-f_\infty\|_{L^2_{\myw}}$ to define a monotonically
decreasing energy functional $\cE[f]$ which is equivalent to
$\|f-f_\infty\|_{L^2_{\myw}}$. Then  we show exponential decaying
directly on the new functional  $\cE[f]$ to prove Theorem \ref{th:1}.

\subsection{Toward the modified energy $\cE[f]$}
We aim to modify the functional $\|f-f_\infty\|_{L^2_{\myw}}^2$ in
order to construct a monotonically decreasing energy functional
$\cE[f]$ . The first step is to characterize the lack of coercivity on
the estimate provided in Proposition \ref{prop:Diss}.  Indeed, thanks
to the Fokker-Planck operator,  we get a dissipation with
respect to quantity $\mathcal I[f](t)$. Hence, the goal is now to modify the functional
$\|f(t)-f_\infty\|_{L^2_{\myw}}^2$ to get a dissipation with
respect to quantity $\|N(t)-N_\infty\|_{L^2_\mybw}^2$. To this aim, we begin with a preliminary
result by considering the following elliptic equation for a given function $S\in L^2(\T\times\R)$ such that
\begin{equation}
  \label{compatibility:g}
  \int_{\T} S(\theta,\nu)\,\dD \theta \,=\, 0\,.
\end{equation}
We consider
\begin{equation}
  \label{eq:v}
  \left\{\begin{array}{l}
    \ds\partial_{\theta}^2 v \,=\, S\,,
    \\[0.9em]
    \ds\int_{\T} v \, \dD \theta = 0\,,
  \end{array}\right.
\end{equation}
and provide some intermediate results on the solutions $v$ to
\eqref{eq:v}.

\begin{lemma}
  \label{lem:1}
  Consider any $S \in L^2_\mybw(\T\times\R)$ which meets condition \eqref{compatibility:g} and $v$ the
  corresponding solution to \eqref{eq:v}. Then, $v$ satisfies the following estimate
  \begin{equation}
    \label{lem:110}
    \|\partial_\theta\, v \|_{L^2_\mybw} \,\leq\, C_P\,\|S\|_{L^2_\mybw}\,,
  \end{equation}
  and
  \begin{equation}
    \label{lem:111}
    \|\partial_{\theta}^2\, v\|_{L^2_\mybw} \,=\,\|S\|_{L^2_\mybw}\,,
  \end{equation}
  where $C_P$ is the Poincaré-Wirtinger constant in
  $$
    \|v\|_{L^2_\mybw} \,\leq\,C_P\,\| \partial_\theta v\|_{L^2_\mybw}\,.
  $$
  Moreover, considering now $v$ the solution to \eqref{eq:v} with source
  term $S\,=\, N_{\infty} - N$, where $N$ is given by \eqref{def:N}. Then it holds for all time $t\,\geq\,0$ that
  \begin{equation}
    \label{bug}
    \partial_{t}\partial_{\theta} v \,+\, \nu\,  \partial_\theta^2 v  \,=\, J \,-\,  \frac{1}{2\pi}\int_\T J\dD\theta\,.
  \end{equation}
\end{lemma}


\begin{proof}
  The first estimate \eqref{lem:110} is obtained by testing the elliptic equation
  \eqref{eq:v} against $v$ and after an integration by part
  \[
    \|\partial_\theta v \|_{L^2_\mybw}^2 \,\leq\,\| S\|_{L^2_\mybw}\,\|v\|_{L^2_\mybw}\,,
  \]
  from which the Wirtinger-Poincar\'e inequality  yields
  $$
    \|\partial_\theta\,v\|_{L^2_\mybw} \,\leq\,C_P\,\| S\|_{L^2_\mybw}\,.
  $$

  The second estimates \eqref{lem:111} directly follows from the equation on $v$
  \eqref{eq:v} and using that $S\in L^2_\mybw$.

  Now let us consider that $N$ is given by \eqref{def:N} and  $S =
    N_\infty-N$. We differentiate in time the
  elliptic equation \eqref{eq:v} and use the equation
  \eqref{eq:momentsCoupled} on $N$ to get
  $$
    \partial_{\theta}^2\partial_t v \,=\, \partial_\theta \left(J+\nu
    N\right) \,.
  $$
  Using again  \eqref{eq:v} and since $N_\infty$ does not depend on $\theta$, it follows that
  $$
    \partial_{\theta}\left( \partial_t\partial_{\theta} v \,+\, \nu\,
    \partial_\theta^2 v \right)  \,=\,\partial_\theta J,
  $$
  hence, there exists $C(t,\nu)$ such that
  $$
    \partial_{t}\partial_\theta v \,+\, \nu\,
    \partial_\theta^2 v  \,=\, J \,+\,  C(t,\nu)\,.
  $$
  Integrating the latter equation in $\theta\in\T$ and using periodic
  boundary condition, we obtain that
  $$
    C(t,\nu) \,=\, -\frac{1}{2\pi}\,\int_{\T} J(t,\theta,\nu)\,\dD\theta,
  $$
  and the result follows.
\end{proof}

{\begin{remark} \label{remark_eps2}
Similar to Remark \ref{remark_eps1}, for the latter usage in Section \ref{sec:diff}, we note that the above lemma still holds for the case of $g(\nu) = \delta_0(\nu)$. For a given function $S\in L^2(\T)$ satisfying
\begin{equation*}
	\int_{\T} S(\theta)\,\dD \theta \,=\, 0\,,
\end{equation*}
let $v = v(\theta)$ be the solution to \eqref{eq:v}. Then, we have
\begin{align*}
\|\partial_\theta\, v \|_{L^2} \,\leq\, C_P\,\|S\|_{L^2}\,.
\end{align*}
We also omit the proof since the argument is identical.
\end{remark}}

Now, we are ready to modify the functional
$\|f(t)-f_\infty\|_{L^2_{\myw}}^2$ and  define a new functional $\cE[f]$ as
\begin{equation*}
  \cE[f](t) \,:=\, \frac{1}{2}\,\|f(t)-f_\infty\|_{L^2_{\myw}}^2 \,+\, \alpha \,   \cA(t)\,,
\end{equation*}
where $\alpha$ is a small parameter to be determined later and  $\cA$ is given by
$$
  \cA(t) \,=\,  -\int_{\T\times \R}  \partial_\theta J(t)\, v(t) \,\mybw\,\dD\nu \dD \theta \,=\,  \int_{\T\times \R}  J(t)\,\partial_\theta v(t) \,\mybw\,\dD\nu \dD \theta\,.
$$
Here, $v$ is the solution to \eqref{eq:v} with source term $S\,=\,
N_{\infty} - N$.

The last step consist in showing that this modified functional
$\cE[f]$ is equivalent to $\|f(t)-f_\infty\|_{L^2_{\myw}}^2$ as long as $\alpha>0$ is sufficiently small.
\begin{lemma}
  \label{lem:equivL2}
  Suppose that $\alpha>0$ satisfies
  \begin{equation}
    \label{cond1:alpha}
    \alpha\,\sqrt{\tsigma}\,C_P \leq \frac{1}{2}.
  \end{equation}
  Then we have
  $$
    \frac{1}{4}\,\|f(t)-f_\infty\|_{L^2_\myw}^2 \,\leq\, \cE[f](t) \,\leq\, \frac{3}{4}\,\|f(t)-f_\infty\|_{L^2_\myw}^2.
  $$
\end{lemma}
\begin{proof}
  We use the Cauchy-Schwarz inequality and Lemma \ref{lem:1} to have
  $$
    |\alpha\,\cA | \,\leq \,
    \alpha\, \|J\|_{L^2_\mybw} \, \|\partial_\theta\,v\|_{L^2_\mybw} \,\leq \,
    \alpha\, C_P\, \|J\|_{L^2_\mybw} \, \|N(t)-N_\infty\|_{L^2_\mybw},
  $$
 {and from \eqref{inequ:J} and \eqref{equ:fInf} we obtain}
  \begin{eqnarray*}
    |\alpha\,\cA | &\leq&
    \alpha\,\sqrt{\tsigma}\,C_P \, \|f(t) - N(t) \,\cM \|_{L^2_{\myw}} \, \|N(t)-N_\infty\|_{L^2_\mybw} \\
    &\leq& \frac{\alpha}{2} \,\sqrt{\tsigma}\,C_P \, \left( \|f(t) - N(t) \,\cM \|_{L^2_{\myw}}^2 + \|N(t)-N_\infty\|_{L^2_\mybw}^2 \right) \\
    &=& \frac{\alpha}{2} \,\sqrt{\tsigma}\,C_P \, \|f(t)-f_\infty\|_{L^2_\myw}^2.
  \end{eqnarray*}
  Hence when $\alpha>0$ satisfies \eqref{cond1:alpha}, we get the expected estimate.
\end{proof}


\subsection{Proof of Theorem \ref{th:1}}
The goal is now to show that when $\alpha$ is sufficiently small,
energy functional $\cE[f]$ is dissipated, which will give the
asymptotic behavior of  $\|f(t)-f_\infty\|_{L^2_{\myw}}$. We have
$$
  \frac{\dD \cA}{\dD t}(t) \,=\, \cI_1(t) \,+\,   \cI_2(t)\,,
$$
with
$$
  \left\{\begin{array}{l}
    \ds\cI_1(t)  \,=\, \int_{\T\times \R} \partial_tJ(t) \,\partial_\theta v(t) \, \mybw\,\dD\nu \dD \theta\,,
    \\[0.9em]
    \ds\cI_2(t) \,=\,  \int_{\T\times \R}  J(t)\,\partial_\theta \partial_t v(t) \,\mybw\,\dD\nu \dD \theta\,.
  \end{array}\right.
$$
We first compute the  term $\cI_1$ using
\eqref{eq:momentsCoupled} which gives
\begin{eqnarray*}
  \cI_1   &=& -\int_{\T\times \R} \partial_\theta\left( P\,+\,\nu\,
  J\right)\,\partial_\theta v \, \mybw\,\dD\nu\dD \theta \\[0.9em]
  && -\, \tkappa\,\int_{\T\times \R} (\sin\ast\rho) \, N\,\partial_\theta v \,\mybw\;\dD\nu \dD \theta\,-\, \frac{1}{m}\,\int_{\T\times \R}  J\,\partial_\theta v \,\mybw\,\dD\nu \dD \theta\,.
\end{eqnarray*}
On the one hand,  integrating by part and using the equation \eqref{eq:v}  on $v$,  the first term on the right hand side can be written as
\begin{eqnarray*}
  \int_{\T\times \R} \partial_\theta\left( P\,+\,\nu\,
  J\right)\,\partial_\theta v \,  \mybw\,\dD\nu\dD \theta
  &=&\int_{\T\times \R} \partial_\theta\left( P\,-\, \tsigma\, N
  \,+\, \tsigma\,(N-N_\infty) \,+\,\nu
  J \right)\,\partial_\theta v \, \mybw\,\dD\nu \dD \theta
  \\[0.9em]
  &=&     -\, \int_{\T\times \R} \left( P -\tsigma\, N\right)\,\left( N_\infty - N\right) \, \mybw\,\dD\nu \dD \theta \\[0.9em]
  && +\,\tsigma \, \|N- N_\infty \|_{L^2_\mybw}^2\,-\, \int_{\T\times \R} \nu \, J \,\partial_\theta^2 v \, \mybw\,\dD\nu \dD \theta \,,
\end{eqnarray*}
where the desired dissipation in $\|N-N_\infty\|_{L^2_\mybw}^2$
now appears. On the other hand, the second term is estimated as in the proof of
Proposition \ref{prop:Diss}, which yields
\begin{eqnarray*}
  \tkappa\,\int_{\T\times \R} (\sin\ast\rho) \, N\,\partial_\theta v \, \mybw\,\dD\nu\, \dD \theta &=&  \tkappa\,\int_{\T\times \R} (\sin\ast\rho) \, (N-N_\infty)\,\partial_\theta v \, \mybw\,\dD\nu \,\dD \theta
  \\[.5em]
  && +\,\tkappa\,\int_{\T\times \R} \left( \sin\ast(\rho-\rho_\infty) \right) \, N_\infty\,\partial_\theta v \, \mybw\,\dD\nu\, \dD \theta
  \\[.5em]
  &\leq& C_P\,\tkappa\,\left( 1 + \frac{1}{2\sqrt\pi}\,\|g\|_{L^2_\mybw}\right)\,\|N-N_\infty\|_{L^2_\mybw}^2 \,.
\end{eqnarray*}
Therefore, gathering the latter computations and  applying Lemmas
\ref{cs} and \ref{lem:1}, we estimate the term $\cI_1$ as 
\begin{eqnarray*}
  \cI_1 &\leq&  -\left[\tsigma-  C_P\,\tkappa\,\left(
    1 +
    \frac{1}{2\sqrt\pi}\,\|g\|_{L^2_\mybw}\right)\right]\,\|N- N_\infty
  \|_{L^2_\mybw}^2 \\[0.9em]
  && +\,   \sqrt{\tsigma} \,\left[
    \frac{C_P}{m}\,+\,
    \,\sqrt{3\,\tsigma}\,\right]\, \|
  N-N_\infty\|_{L^2_\mybw} \,\sqrt{\cI[f]}(t)\,+\int_{\T\times\R} \nu\,J
  \,\partial_\theta^2 v \,\mybw\,\dD\nu\,\dD\theta\,.
\end{eqnarray*}
Now we use the term $\cI_2$ to remove the last term on the right
hand side of the latter equation. Indeed, using the equation
\eqref{bug} in Lemma \ref{lem:1}, we have
\begin{eqnarray*}
  \cI_2 &=&  \int_{\T\times\R} J\,\left( J -
  \frac{1}{2\pi}\int_{\T}J(\theta')\,\dD\theta' \,-\,
  \nu\,\partial_\theta^2 v \right)\,\mybw\,\dD\nu\,\dD\theta\,,
  \\
  &\leq& 2\,  \|J\|_{L^2_\mybw}^2
  \,-\,\int_{\T\times\R} \nu\, J
  \partial_\theta^2 v \,\mybw\,\dD\nu\,\dD\theta\,,
  \\
  &\leq& 2\,\tsigma \, \cI[f]
  \,-\,\int_{\T\times\R} \nu\, J
  \partial_\theta^2 v \,\mybw\,\dD\nu\,\dD\theta\,,
\end{eqnarray*}
Gathering the latter results, we obtain
\begin{eqnarray*}
  \frac{\dD \cA}{\dD t}(t) &\leq&  2\,\tsigma\,\cI[f](t)  \,-\,  \left[\tsigma\,-\,  C_P\,\tkappa\,\left(
    1 +
    \frac{1}{2\sqrt\pi} \,\|g\|_{L^2_\mybw}\right)\right]\,\|N(t)- N_\infty
  \|_{L^2_\mybw}^2 \\
  && +\,    \sqrt{\tsigma}\,\left[ \frac{C_P}{m}\,+\, \sqrt{3\,\tsigma}\right]\, \| N(t)-N_\infty\|_{L^2_\mybw} \,\sqrt{\cI[f]}(t)\,.
\end{eqnarray*}
Applying the Young inequality to the last term with $\eta>0$ such that
$$
  \eta \, \left[ \frac{C_P}{m}\,+\, \sqrt{3\,\tsigma}\right]^2 \,=\, 1,
$$
or equivalently,
$$
  \frac{\eta}{2} \,\tsigma\, \left[ \frac{C_P}{m}\,+\, \sqrt{3\,\tsigma}\right]^2 \,=\, \frac{\tsigma}{2}\,,
$$
it yields that
\begin{eqnarray*}
  \frac{\dD \cA}{\dD t}(t) &\leq&
  \left[2\,+\, \frac12 \left(\frac{C_P}{m\sqrt{\tsigma}}+\sqrt{3}\right)^2\right]\, \tsigma\,\cI[f](t)
  \\
  && -\,  \left[\frac{\tsigma}{2}-  C_P\,\tkappa\,\left(
    1 +
    \frac{1}{2\sqrt\pi}\,\|g\|_{L^2_\mybw}\right)\right]\,\|N(t)- N_\infty
  \|_{L^2_\mybw}^2\,.
\end{eqnarray*}
Finally, from the definition of $\cE[f]$ and applying Proposition \ref{prop:Diss} with the latter inequality, we get  that for any $\alpha>0$,
\begin{eqnarray*}
  \frac{\dD\cE[f]}{\dD t}(t) &\leq& -\left[\frac{\tsigma}{m}  \,-\,
    \frac{\tkappa}{2}
    \,\left(3+\frac{1}{2\sqrt\pi}\,
    \|g\|_{L^2_\mybw}\right)\right]
  \,\cI[f](t)
  \\[0.9em]
  && -\,  \alpha\,\left[\frac{\tsigma}{2}\,-\,
    C_P\,\tkappa\,\left(1
    +\frac{1}{2\sqrt\pi}\,\|g\|_{L^2_\mybw}\right)\right]\,\|N(t)-
  N_\infty\|_{L^2_\mybw}^2
  \\[0.9em]
  && +\,
  \frac{\tkappa}{2}
  \left(1+\frac{1}{2\sqrt\pi}\, \|g\|_{L^2_\mybw}\right)\,\|N(t)-N_\infty\|^2_{L^2_\mybw}
  \\[0.9em]
  &&+\,\alpha\,
  \left[2\,+\, \frac12 \left(\frac{C_P}{m\,\sqrt{\tsigma}}+\sqrt{3}\right)^2\right]\,
  \tsigma\,                           \cI[f](t)\,.
\end{eqnarray*}
We choose $\alpha>0$ such that the last term in the previous inequality
becomes sufficiently small to be absorbed by the first term, that is,
taking $\alpha>0$ such that
$$
  \alpha\,m\, \left[ 2 \,+\,
    \frac12 \left(\frac{C_P}{m\,\sqrt{\tsigma}} + \sqrt3\right)^2\right] \,\leq\, \frac{1}{2},
$$
for instance
\begin{align} \label{C-1-2}
{\alpha \,=\, \min\left( \frac{1}{2\,C_P} \,, \,\frac {m\,\tsigma}{2\left(5 \,m^2 \,\tsigma \,+\,C_P^2\right)} \,, \,\frac{1}{2\,C_P\,\sqrt{\tsigma}} \, \right)\,,}
\end{align}
we have
\begin{eqnarray}
  \frac{\dD \cE[f]}{\dD t}(t) &\leq&\ds -\left[\frac{\tsigma}{2\,m}  \,-\,
    \frac{\tkappa}{2}
    \,\left(3+\frac{1}{2\sqrt\pi}\,
    \|g\|_{L^2_\mybw}\right)\right]
  \,\cI[f](t) \nonumber
  \\[0.9em]
  &&\ds -  \,\alpha\,\left[\frac{\tsigma}{2}-
    C_P\,\tkappa\,\left(1
    +\frac{1}{2\sqrt\pi}\,\|g\|_{L^2_\mybw}\right)\right]\,\|N(t)-
  N_\infty\|_{L^2_\mybw}^2 \nonumber
  \\[0.9em]
  &&\ds + \,
  \frac{\tkappa}{2}\,
  \left(1+\frac{1}{2\sqrt\pi}\,
  \|g\|_{L^2_\mybw}\right)\,\|N(t)-N_\infty\|^2_{L^2_\mybw} \label{C-1-1}
  \\[0.9em]
  &\leq&\ds  -\left[\frac{\tsigma}{2\,m}  \,-\,
    \frac{\tkappa}{2}
    \,\left(3+\frac{1}{2\sqrt\pi}\,
    \|g\|_{L^2_\mybw}\right)\right]
  \,\cI[f](t) \nonumber
  \\[0.9em]
  && \ds -  \,\alpha\,\left[\frac{\tsigma}{2}\,-\,
    \frac{\tkappa}{\alpha}\,\left(1 
    +\frac{1}{2\sqrt\pi}\,\|g\|_{L^2_\mybw}\right)\right]\,\|N(t)-
  N_\infty\|_{L^2_\mybw}^2. \nonumber
\end{eqnarray}
Finally, {when $\tkappa>0$ is sufficiently small relative to $\tsigma>0$,} that is,
\begin{align} \label{C-1-3}
  \left\{
  \begin{array}{l}
    \ds \frac{\tkappa}{\alpha} \, \left(1
    +\frac{1}{2\sqrt\pi}\,\|g\|_{L^2_\mybw}\right) \leq
    \frac{\tsigma}{4}\,,
    \\[0.9em]
    \ds    \tkappa \,\left(3
    +\frac{1}{2\sqrt\pi}\,\|g\|_{L^2_\mybw}\right) \leq
    \frac{\tsigma}{2\,m}\,,
  \end{array}\right.
\end{align}
then the right hand side of the latter estimate corresponds to a dissipation of both $\cI[f](t)$ and $\|N(t)-N_\infty\|_{L^2_\mybw}^2$. With \eqref{C-1-2}, diffusive regime \eqref{C-1-3} can be reformulated {as
\begin{eqnarray*}
	\tkappa \,\left(3+\frac{1}{2\sqrt\pi}\,\|g\|_{L^2_\mybw}\right) &\leq& \min\left( \frac{\alpha\,\tsigma}{4} \,,\, \frac{\tsigma}{2\,m} \, \right) \\[.5em]
	&=& \min\left( \frac{\tsigma}{8\,C_P} \,, \,\frac {m\,\tsigma^2}{8\left(5 \,m^2 \,\tsigma \,+\,C_P^2\right)} \,, \,\frac{\sqrt{\tsigma}}{8\,C_P} \,,\, \frac{\tsigma}{2\,m} \, \right).
\end{eqnarray*}
Furthermore, one can solve
\begin{align*}
\tkappa \,\left(3+\frac{1}{2\sqrt\pi}\,\|g\|_{L^2_\mybw}\right) \,\leq\, \frac {m\,\tsigma^2}{8\left(5 \,m^2 \,\tsigma \,+\,C_P^2\right)}
\end{align*}
to obtain
\begin{align*}
p_1\,m\,\tkappa \,+\,\sqrt{p_2\,m^2\,\tkappa^2 \,+\, \frac{p_3\,\tkappa}{m}} \,\leq\, \tsigma
\end{align*}
for some positive constants $p_1$, $p_2$, and $p_3$ only depending on $\|g\|_{L^2_\mybw}$. Hence, we summarize that there exists a constant $C_\infty>0$, only depending on $\|g\|_{L^2_\mybw}$, such that $\tkappa>0$ and $\tsigma>0$ satisfy
$$
C_\infty \,\max\left( \sqrt{\frac{\tkappa}{m}}\,, \,\tkappa\,, \,m\tkappa \,, \,\tkappa^2 \, \right) \,\leq\, \tsigma,
$$
which is \eqref{B12_bis}. Then, by applying \eqref{C-1-3}, Lemma \ref{lem:GPoincare}, and Lemma \ref{lem:equivL2} on \eqref{C-1-1}, we have
\begin{eqnarray*}
\frac{\dD\cE[f] }{\dD t}(t) &\leq& -\,\frac{\tsigma}{4\,m}\,\cI[f](t) \,-\,\frac{\alpha\,\tsigma}{4}\, \|N(t)-
N_\infty\|_{L^2_\mybw}^2 \\[.5em]
&\leq& -\min\left( \frac{\tsigma}{4\,m}\,, \,\frac{\alpha\,\tsigma}{4}\, \right) \,\|f(t)-f_\infty\|_{L^2_\myw}^2 \\[.5em]
&\leq& -\,\frac43\,\min\left( \frac{\tsigma}{4\,m}\,, \,\frac{\alpha\,\tsigma}{4}\, \right) \,\cE[f] \,=:\, -\,C\,\cE[f]\,,
\end{eqnarray*}
where $C>0$ only depends on $\tsigma$ and $m$ by \eqref{C-1-2}:
\begin{align*}
C \,=\, \frac13\,\min\left( \frac{\tsigma}{m}\,, \, \alpha\,\tsigma \,\right) \,=\, \frac13\,\min\left( \frac{\tsigma}{m}\,, \,\frac{\tsigma}{2\,C_P} \,, \,\frac {m\,\tsigma^2}{2\left(5 \,m^2 \,\tsigma \,+\,C_P^2\right)} \,, \,\frac{\sqrt{\tsigma}}{2\,C_P} \, \right) \,.
\end{align*}}
Again since $\cE[f]$ and
$\|f-f_\infty\|_{L^2_\myw}$ are equivalent (Lemma \ref{lem:equivL2}),
we get the expected results using the Gronwall's inequality, which completes our proof.

{\begin{remark}
Note that throughout this section, all nonlinear terms including the sine interaction kernel are bounded in absolute values, for example, \eqref{absBound}, and following properties of sine function are used:
\begin{align} \label{kernel}
\sin\in C^\infty(\mathbb T)\cap L^\infty(\mathbb T) \cap L^2(\mathbb T) \quad \mbox{and} \quad \int_{\mathbb T} \sin\theta \,\dD\theta \,=\,0\,.
\end{align}
That is, the precise structure of the interaction kernel is not essential to the argument. Therefore, the sine interaction kernel can be replaced by any other suitable interaction kernel satisfying \eqref{kernel}, and similar results can be obtained by the same method in the regime of sufficiently small coupling strength relative to the noise intensity, i.e., $\tsigma\gg\tkappa$.
\end{remark}
}

\section{Diffusion limit when $g=\delta_0$}
\label{sec:diff}
\setcounter{equation}{0}

We now suppose that all oscillators are identical, that is, take
$g=\delta_0$. It allows to remove the $\nu$ variable in the
previous system, which now becomes \eqref{K4-1}-\eqref{K4-2}.  Thus,
we consider the diffusion limit to prove Theorem \ref{th:2} by
propagating some regularity in $\theta$.

%
%

{\subsection{Outline of the proof of Theorem \ref{th:2}}
We proceed as in the proof of Theorem \ref{th:1} on the long time
behavior but now estimate
$\|f^\eps(t)-\rho^\eps(t)\,\cM\|_{L^2_{\cM^{-1}}}$, which yields that
\begin{align*}
\frac12 \frac{\dD}{\dD t} \|f^\eps(t)-\rho^\eps(t)\,\cM\|_{L^2_{\cM^{-1}}}^2 \,\lesssim\, -\frac{1}{\eps^2}\, \cI[f^\eps](t) \,+\, \left( \|f^\eps\|_{L^2_{\cM^{-1}}}^2\,+\,\|\partial_\theta f^\eps\|_{L^2_{\cM^{-1}}}^2 \right),
\end{align*}
where $\cI[f^\eps](t)\geq 0$ is again  the dissipation  associated with  the
Fokker-Planck operator for the weighted $L^2_{\cM^{-1}}$
norm.
This requires the propagation of the weighted $L^2_{\cM^{-1}}$ norm
for $f^\eps$ and $\partial_\theta f^\eps$ uniformly with respect to
$\eps$, which will be shown in Proposition
\ref{prop:L2}. \\
Unfortunately the previous estimate  does not directly
control   $f^\eps(t)-\rho(t)\,\cM$,  the goal is then to modify this functional to
obtain a new dissipation estimate that controls
$\|\rho(t)-\rho^\eps\|_{L^2}$. Again as in the proof of Theorem
\ref{th:1}, a  key observation  is that as $f^\eps(t)$ approaches
$\rho^\eps \cM$, as $\eps\rightarrow 0$, the macroscopic
quantity $P^\eps$,  defined in \eqref{macro_eps},  satisfies
$P^\eps-\tilde\sigma \rho^\eps$ converges to zero.  Consequently,   the equation
\eqref{eq:macro}  governing the density and 
momentum $(\rho^\eps, J^\eps)$ can be rewritten to reflect this asymptotic behavior as
$$
\partial_t \left(\rho^\eps -\eps \partial_\theta J^\eps\right)  \,-\, \partial_\theta\biggl(  \tilde\sigma\partial_\theta \rho^\eps  \,+\,
  \tkappa\,(\sin\ast\rho^\eps)\, \rho^\eps \biggr) \,=\,
  \partial^2_\theta \left(P^\eps-\tilde\sigma \rho^\eps\right).
  $$
  Then the quantity $\rho^\eps -\eps \partial_\theta J^\eps$  will be
  compared to  the limit density $\rho$ satisfying \eqref{eq:dd}.  To this end, we select  $\cA(t)$ such that:
\begin{equation}
  \label{def:A2}
 \cA(t) \,:=\,  - \frac{1}{2}\,\int_{\T}  \left( \rho(t)-\rho^\eps(t) +\eps\partial_\theta J^\eps(t)\right)\,v(t)
 \, \dD \theta
 \end{equation}
 where $v(t)$ satisfies:
 $$
 \partial_{\theta}^2 v \,=\,  \rho(t)-\rho^\eps(t) +\eps\partial_\theta J^\eps(t),
 $$
 so that
 $$
\cA(t) \,:=\,   \frac{1}{2}\,\|\partial_\theta v^\eps(t)  \|_{L^2}^2\,,
 $$
which is equivalent to $\left\| \rho(t)-\rho^\eps(t)  +\eps\partial_\theta J^\eps(t)\right\|_{H^{-1}}$.
Finally,  for a sufficiently small $\eps$, we get that 
\begin{align*}
\frac{\dD \cA }{\dD t}(t) \,\lesssim\, - \| \rho^\eps(t)-\eps\,\partial_\theta J^\eps(t)-\rho(t)\|_{L^2}^2 \,+\,\left( \cA(t)+\eps^2
\,\|\partial_\theta f^\eps
\|_{L^2_{\cM^{-1}}}^2 \,+\,  \|f^\eps -
\rho^\eps \cM \|_{L^2_{\cM^{-1}}}^2\right),
\end{align*}
Once again, the propagation of the weighted $L^2_{\cM^{-1}}$ norm for $\partial_\theta f^\eps$ uniformly with respect to $\eps$ is required.
}

\subsection{Basic estimates}

First, we aim to study the propagation of the weighted $L^2_{\cM^{-1}}$ norm for
$f^\eps$ and $\partial_\theta f^\eps$ uniformly with respect to
$\eps$. We prove the following preliminary result.

\begin{proposition}
  \label{prop:L2}
  Let $f^\eps \,=\, f^\eps(t, \theta, \omega)$ be a classical solution
  to \eqref{K4-1}-\eqref{K4-2} such that the initial data satisfy
  \eqref{hyp:th:2}.  Then, we have
\begin{equation}
    \label{res:1}
   \frac12 \frac{\dD}{\dD t} \|f^\eps(t)\|_{L^2_{\cM^{-1}}}^2  \,\leq\, \frac{\tkappa^2}{2\,\tsigma}\, \|f^\eps(t)\|_{L^2_{\cM^{-1}}}^2
    \,-\,\frac{\tsigma}{2\,\eps^2} \,\cI[f^\eps](t)\,,
  \end{equation}
  {where $\cI[f^\varepsilon](t)$ corresponds to the dissipation of the Fokker-Planck operator and is defined as
  \begin{align*}
  	\cI[f^\eps](t) \,:=\, \int_{\T\times \R} \cM(\omega) \, \left| \partial_\omega
  	\left( \frac{f^\eps(t)}{\cM} \right) \right|^2 \,\dD \bz \,\geq\, 0.
  \end{align*}}
Moreover,   for all time $t\geq 0$
  \begin{equation}
    \label{res:2}
   \|f^\eps(t)\|_{L^2_{\cM^{-1}}} \,+\,   \|\partial_\theta
   f^\eps(t)\|_{L^2_{\cM^{-1}}}  \,\leq\, \left(  \|\partial_\theta
   f^\eps_{\rm in}\|_{L^2_{\cM^{-1}}} \,+\, 3\, \|
   f^\eps_{\rm in}\|_{L^2_{\cM^{-1}}}  \right) \, e^{C\,t}\,,
\end{equation}
where  $C=\tkappa^2/\tsigma$.
\end{proposition}


\begin{proof}
  We first proceed as in Proposition \ref{prop:Diss}, but combine the
  terms in a different way to get uniform estimates with respect to
  $\eps$.  We consider the centred Gaussian distribution $\cM$ and  multiply
  \eqref{K4-1} by $f^\eps \cM^{-1}$, then we integrate with
  respect to $\bz:=(\theta,\omega)\in\T\times\R$, it yields 
 \begin{eqnarray*}
    \frac12 \frac{\dD}{\dD t} \|f^\eps(t)\|_{L^2_{\cM^{-1}}}^2
      & =& -\,\frac{\tkappa}{\eps}\,\int_{\T\times \R} (\sin\ast\rho^\eps(t)) \, f^\eps(t)\,
    \partial_\omega \left( \frac{f^\eps(t)}{\cM} \right)\,\dD \bz
    \,-\, \frac{\tsigma}{\eps^2} \,\cI[f^\eps](t)\,.
 \end{eqnarray*}
 
 It is left to estimate the first term of the right hand side as
 \begin{eqnarray*}
    \left| \,\frac{\tkappa}{\eps}\,\int_{\T\times \R} (\sin\ast\rho^\eps)\, f^\eps\, \partial_\omega \left(
   \frac{f^\eps}{\cM} \right) \, \dD\bz\, \right|  & \leq&
                                            \frac{\tkappa^2}{2\,\eta}\,\|
                                            \sin\ast\rho^\eps\|_{L^\infty}^2\,
                                            \|f^\eps\|_{L^2_{\cM^{-1}}}^2
                                            \,+\,
                                            \frac{\eta}{2\,\eps^2}\,\cI[f^\eps]\,,
 \end{eqnarray*}
 where $\eta>0$ is a free parameter to be defined later.  Using the Young's
 convolution inequality and the conservation of mass, we have
 $$
 \|\sin\ast\rho^\eps(t)\|_{L^\infty}\,\leq\, \|\rho^\eps(t)\|_{L^1} \,=\,
 \|f^\eps(t)\|_{L^1}\,=\, 1
 $$
and choosing $\eta=\tsigma$, it gives the first estimate \eqref{res:1}
 $$
 \frac12 \frac{\dD}{\dD t} \|f^\eps(t)\|_{L^2_{\cM^{-1}}}^2  \,\leq\, \frac{\tkappa^2}{2\,\tsigma}\, \|f^\eps(t)\|_{L^2_{\cM^{-1}}}^2
    \,-\,\frac{\tsigma}{2\,\eps^2} \,\cI[f^\eps](t).
 $$
Finally from the Gronwall's inequality, we get that
\begin{equation}
  \label{resu:01}
\|f^\eps(t)\|_{L^2_{\cM^{-1}}} \,\leq\,   \|f^\eps_{\rm
    in}\|_{L^2_{\cM^{-1}}} \, e^{C_0\,t}\,,
 \end{equation}
    with $C_0 = \tkappa^2/(2\,\tsigma)$.

    Then we set $h^\eps=\partial_\theta f^\eps$ and differentiate
    \eqref{K4-1}-\eqref{K4-2} with respect to $\theta$, it yields the
    following equation
    $$
    \eps\,\partial_t h^\eps \,+\,\omega \,\partial_\theta h^\eps \,-\,\tkappa\,
  (\sin\ast\rho^\eps) \,\partial_\omega h^\eps \,=\, \frac{1}{\varepsilon} \, \cL_{\mathrm{FP}}[h^\eps]\,+\,\tkappa\,  (\cos\ast\rho^\eps)  \,\partial_\omega f^\eps\,,
  $$
  which has the same structure as the equation on $f^\eps$ with the
  additional source term $\tkappa\,(\cos\ast\rho^\eps)\,\partial_\omega
  f^\eps$. Hence, proceeding as previously, we now obtain
  $$
 \frac12 \frac{\dD}{\dD t} \|h^\eps(t)\|_{L^2_{\cM^{-1}}}^2  \,\leq\,
 \frac{\tkappa^2}{\tsigma}\, \left(
   \|h^\eps(t)\|_{L^2_{\cM^{-1}}}^2 \,+\,  \|f^\eps(t)\|_{L^2_{\cM^{-1}}}^2 \right)
   \, -\,\frac{\tsigma}{2\,\eps^2} \,\cI[h^\eps](t).
    $$
    Then, using \eqref{resu:01}, {we get
    \begin{align*}
    \frac{\dD}{\dD t} \|h^\eps(t)\|_{L^2_{\cM^{-1}}}^2  \,\leq\, \frac{2\tkappa^2}{\tsigma}\, \left(
    \|h^\eps(t)\|_{L^2_{\cM^{-1}}}^2 \,+\,  \|f^\eps_{\rm in}\|_{L^2_{\cM^{-1}}}^2 \, e^{2C_0\,t} \right)\,,
    \end{align*}
    which implies
    \begin{align*}
    	& \|h^\eps(t)\|_{L^2_{\cM^{-1}}}^2 \, e^{-4C_0\,t} \,-\, \|h^\eps(0)\|_{L^2_{\cM^{-1}}}^2  \\[.5em]
    	& \quad\leq\, 4C_0\, \|f^\eps_{\rm in}\|_{L^2_{\cM^{-1}}}^2 \, \int_0^t e^{-2C_0\,s} \, \dD s \,=\, 2\,\|f^\eps_{\rm in}\|_{L^2_{\cM^{-1}}}^2 \,\left( 1 - e^{-2C_0\,t} \right) \,\leq\, 2\,\|f^\eps_{\rm in}\|_{L^2_{\cM^{-1}}}^2\,.
    \end{align*}}
Hence, we have the second estimate 
    \begin{equation}
      \label{resu:02}
 \|h^\eps(t)\|_{L^2_{\cM^{-1}}} \,\leq \, \left(  \|\partial_\theta
   f^\eps_{\rm in}\|_{L^2_{\cM^{-1}}} \,+\, \sqrt{2}\, \|
   f^\eps_{\rm in}\|_{L^2_{\cM^{-1}}}  \right) \, e^{2\, C_0\,t}\,.
   \end{equation}
 Gathering the latter estimates \eqref{resu:01} and \eqref{resu:02}, we obtain \eqref{res:2}. 
  \end{proof}


 \subsection{Proof of Theorem \ref{th:2}}
From Proposition \ref{prop:L2}, we may now prove our second main
result.  On the one hand, we   evaluate a kind of relative entropy in
the weighted $L^2$ space, 
  \begin{eqnarray*}
    \frac12 \frac{\dD}{\dD t}
    \|f^\eps-\rho^\eps\,\cM\|_{L^2_{\cM^{-1}}}^2  &=& \frac12 \frac{\dD}{\dD t}
    \|f^\eps\|_{L^2_{\cM^{-1}}}^2  \,-\,   \frac12 \frac{\dD}{\dD t} \|\rho^\eps\|_{L^2}^2
                                                           \\[0.9em]
                                                       &\leq&
                                                              \,-\,\frac{\tsigma}{2\,\eps^2}
                                                              \,\cI[f^\eps]
                                                              \,+\, \frac{\tkappa^2}{2\,\tsigma}\, \|f^\eps\|_{L^2_{\cM^{-1}}}^2  \,+\,  \frac{1}{\eps}\,\int_\T \rho^\eps \, \partial_\theta
                                                           J^\eps \dD\theta\,.
  \end{eqnarray*}
  Hence,  after integrating by part and applying the Young inequality
  on  the last term
        on the right hand side of the latter inequality, {we apply Lemma \ref{cs} and Remark \ref{remark_eps1}}, which yields
  \begin{eqnarray*}
  \frac12 \frac{\dD}{\dD t}
    \|f^\eps-\rho^\eps\,\cM\|_{L^2_{\cM^{-1}}}^2  &\leq&  \,-\,\frac{\tsigma}{4\,\eps^2} \,\cI[f^\eps]\,+\, \frac{\tkappa^2}{2\,\tsigma}\, \|f^\eps\|_{L^2_{\cM^{-1}}}^2\,+\, \|\partial_\theta\rho^\eps\|_{L^2}^2\,.
  \end{eqnarray*}
Moreover, observing that
  $$
\|\partial_\theta \rho^\eps \|_{L^2} \,\leq \,\|\partial_\theta f^\eps\|_{L^2_{\cM^{-1}}},
  $$
using again the Gaussian-Poincar\'e inequality with respect to probability
measure $\cM \,\dD\omega$ to have
$$
  \|f^\eps-\rho^\eps\,\cM\|_{L^2_{\cM^{-1}}}^2\,\leq\,\cI[f^\eps],
$$
and the $H^1$ estimate \eqref{res:2} of  Proposition \ref{prop:L2}, we obtain
\begin{eqnarray*}
  \frac12 \frac{\dD}{\dD t}
    \|f^\eps-\rho^\eps\,\cM\|_{L^2_{\cM^{-1}}}^2  &\leq&
                                                         \,-\,\frac{\tsigma}{4\,\eps^2}
                                                         \|f^\eps-\rho^\eps\,\cM\|_{L^2_{\cM^{-1}}}^2\,\\[0.9em]
  &+&
                                                         \max\left(\frac{C}{2},1\right)\, \left(  \|\partial_\theta
   f^\eps_{\rm in}\|_{L^2_{\cM^{-1}}} \,+\, 3\, \|
   f^\eps_{\rm in}\|_{L^2_{\cM^{-1}}}  \right)^2 \,e^{2C\,t}\,,
  \end{eqnarray*}
  hence from the Gronwall's lemma, we get the
  first estimate  of Theorem \ref{th:2}
  \begin{eqnarray}
    \label{relative:H}
     \|f^\eps-\rho^\eps\,\cM\|_{L^2_{\cM^{-1}}} &\leq&
    \|f^\eps_{\rm in}-\rho^\eps_{\rm in}\,\cM\|_{L^2_{\cM^{-1}}}
                                                       e^{-\tsigma\,t/(4\eps^2)} \\[0.9em]
\nonumber
                                                &+&
                                                         \sqrt{2}\, \eps\,\max\left(\frac{\tkappa}{\tsigma},\,\sqrt{\frac{2}{\tsigma}}
                                                           \right)\, \left(  \|\partial_\theta
   f^\eps_{\rm in}\|_{L^2_{\cM^{-1}}} +3\, \|
   f^\eps_{\rm in}\|_{L^2_{\cM^{-1}}}  \right)\,e^{C\,t}\,.
  \end{eqnarray}
On the other hand to prove the convergence of $\rho^\eps$ to its limit $\rho$
given by \eqref{eq:dd}, we  define $\cA(t)$  as
\begin{equation}
  \label{eq:A2}
  \cA(t) \,=\,  \frac{1}{2}\,\|\partial_\theta v^\eps(t) \,
  \|_{L^2}^2\,,
  \end{equation}
where $v$ is now  solution to \eqref{eq:v} with source term
$$
S\,=\,\rho - \rho^\eps +\eps\,\partial_\theta J^\eps\,.
$$
First let us observe that $v^\eps$ is well defined since the
compatibility condition \eqref{compatibility:g} on $S$ is well satisfied. Before proving the second estimate of Theorem \ref{th:2}, let us show
that $\cA(t)$ gives a $H^{-1}$ estimate on $\rho^\eps-\rho$. Indeed, the following Lemma ensures that $\cA(t)$ is
controlled by the squares of the weighted $L^2_{\cM^{-1}}$ norm of $\partial_\theta f^\eps$ and the $H^{-1}$ norm of $\rho^\eps-\rho$.
\begin{lemma}
We consider $\cA(t)$ defined by \eqref{eq:A2}. It holds uniformly with respect to $\eps$
\begin{equation*}
\cA(t)\,\leq\, \|\rho^\eps(t) -\rho(t)\|_{H^{-1}}^2
\,+\,\tsigma\,C_P^2\,\eps^2\,\|\partial_\theta f^\eps(t)\|_{L^2_{\cM^{-1}}}^2\,,
\end{equation*}
and
\begin{equation*}
\frac{1}{4}\,
\|\rho^\eps(t) -\rho(t)\|_{H^{-1}}^2
\,-\,\tsigma\,C_P^2\,\frac{\eps^2}{2}\,
\|\partial_\theta f^\eps(t)\|_{L^2_{\cM^{-1}}}^2
\,\leq\,
\cA(t)
\,.
\end{equation*}
 \label{lem:4}
\end{lemma} 
\begin{proof}
Defining $w^\eps$ and $u^\eps$ as the respective solutions to
\eqref{eq:v} with source term $S\,=\,-\,\partial_\theta J^\eps$ and $\rho-\rho^\eps$, it holds
\[
v^\eps\,=\,u^\eps\,-\,\eps\,w^\eps\,.
\]
We apply operator $\partial_\theta$ to the latter relation, take the $L^2$ norm, and apply the triangular inequality, it yields 
\[
\sqrt{2\,\cA}\,\leq\,
\left\|
\partial_\theta u^\eps\right\|_{L^2}
\,+\, \eps\,
\left\|
\partial_\theta w^\eps
\right\|_{L^2}\,,
\]
and
\[
\left\|
\partial_\theta u^\eps
\right\|_{L^2}
\,-\,\eps\,
\left\|
\partial_\theta w^\eps
\right\|_{L^2}
\,\leq\,\sqrt{2\,\cA}\,.
\]
We estimate 
$
\left\|
\partial_\theta w^\eps
\right\|_{L^2}
$ by applying {Remark \ref{remark_eps2}} with source term
$S\,=\,-\,\partial_\theta J^\eps$ and using that
{
\begin{eqnarray*}
\partial_\theta J^\eps &=& \int_{\R} \omega\,\partial_\theta f(t,\bz)\,\dD \omega \\[.5em]
&\leq& \left( \int_{\R} \omega^2 \,\cM \,\dD \omega \right)^{\frac12} \left( \int_{\R} \left| \partial_\theta f(t,\bz) \right|^2 \,\cM^{-1} \,\dD \omega \right)^{\frac12} \\[.5em]
&=& \sqrt{\tsigma} \, \left( \int_{\R} \left| \partial_\theta f(t,\bz) \right|^2 \,\cM^{-1} \,\dD \omega \right)^{\frac12} \,,
\end{eqnarray*}
which yields}
\[
\sqrt{2\,\cA}\,\leq\,
\|\rho^\eps -\rho\|_{H^{-1}}
\,+\,
\eps\,C_P\,\sqrt{\tsigma}\,\|\partial_\theta f^\eps\|_{L^2_{\cM^{-1}}}\,,
\]
and
\[
\|\rho^\eps -\rho\|_{H^{-1}}
\,-\,
\eps\,
C_P\,\sqrt{\tsigma}\,\|\partial_\theta f^\eps\|_{L^2_{\cM^{-1}}}
\,\leq\,\sqrt{2\,\cA}
\,.
\]
We obtain the result taking the square of the latter inequalities and applying Young's inequality.
\end{proof}

Now let us evaluate $\cA(t)$ observing that
 $$
\frac{\dD \cA }{\dD t}(t) \,=\, \left\langle \partial_t\partial_\theta v^\eps(t),\,\partial_\theta v^\eps(t) \right\rangle \,=\, \left\langle \partial_t\left(\rho^\eps(t)\,-\,\eps\,\partial_\theta J^\eps(t)  \,-\,\rho(t)\right),\,v^\eps(t)\right\rangle.
$$
Therefore, relying on equations \eqref{eq:macro} and \eqref{eq:dd}, we deduce 
\begin{equation}
  \label{toto}
\frac{\dD \cA }{\dD t}(t) \,=\, -\,\tsigma\, \| \rho^\eps(t)-\eps\,\partial_\theta J^\eps(t)-\rho(t)\|_{L^2}^2
\,+\,
\cE_{1}(t)
\,+\,
\cE_{2}(t)\,,
\end{equation} 
where
\begin{equation*}
\left\{
 \begin{array}{l}
 \ds \cE_{1}(t)\,=\, -\left\langle P^\eps(t) - \tsigma\,\left(\rho^\eps(t)
   \,-\,\eps\,\partial_\theta J^\eps(t)\right),\, \rho^\eps(t)
   \,-\,\eps\,\partial_\theta J^\eps(t) - \rho(t)\right\rangle ,
\\[1.1em]
 \ds\cE_{2}(t) \,=\, -\,\tkappa\,\left\langle \sin\ast\rho^\eps(t) \,\rho^\eps(t)
   - \sin\ast\rho(t) \,\rho(t), \,\partial_\theta v^\eps(t)\right\rangle\,. 
 \end{array}\right.
\end{equation*}
{First noting that from Remark \ref{remark_eps1}
$$
\| P^\eps - \tsigma \,\rho^\eps \|_{L^2} \,\leq \, \sqrt{3}\,\tsigma\,    \, \| f^\eps -
\rho^\eps \cM \|_{L^2_{\cM^{-1}}}\,, 
$$}
we have for any $\eta_1>0$
$$
\cE_{1}(t) \leq \frac{3 \,\tsigma^2}{\eta_1}\,    \, \| f^\eps(t) -
\rho^\eps(t) \cM \|_{L^2_{\cM^{-1}}}^2 \,+\, \frac{\eps^2 \,\tsigma^2}{\eta_1}\,\, \|\partial_\theta J^\eps (t)  \|_{L^2}^2  \,+\,  \frac{\eta_1}{2} \, \| \rho^\eps(t)-\eps\,\partial_\theta J^\eps(t)-\rho(t)\|_{L^2}^2\,.
$$
Then we evaluate the term $\cE_2$ as follows
$$
\cE_2(t) \,=\, \cE_{21}(t) +
\cE_{22}(t)\,,
$$
where
\begin{equation*}
\left\{
 \begin{array}{l}
 \ds \cE_{21}(t) \,=\, -\,\tkappa\,\left\langle \sin\ast\rho \,\left(
   \rho^\eps \,-\,  \rho\right)(t), \,\partial_\theta
   v^\eps(t)\right\rangle\,, 
\\[1.1em]
 \ds\cE_{22}(t)  \,=\, -\,\tkappa\,\left\langle \sin\ast\left(\rho^\eps\,-\,\rho\right)(t)\,\rho^\eps(t), \,\partial_\theta v^\eps(t)\right\rangle\,. 
 \end{array}\right.
\end{equation*}
Again applying the Young's inequality, we have for any $\eta_{21} > 0$,
\begin{eqnarray*}
\cE_{21}(t)  &\leq & \tkappa\, \|\rho_{\rm in}\|_{L^1}\,\biggl( 
\|\rho^\eps (t) \,-\,\rho (t)  \,-\, \eps\,\partial_\theta J^\eps (t)  \|_{L^2}\,+\,
\eps\,\|\partial_\theta J^\eps (t)  \|_{L^2} \biggr)\, \|\partial_\theta
                  v^\eps (t) \|_{L^2}\,,
  \\[0.9em]
          &\leq &
\frac{\eta_{21}}{2}\, \left(\|\rho^\eps (t) \,-\,\rho (t)  \,-\, \eps\,\partial_\theta
                  J^\eps (t)  \|_{L^2}^2 \,+\, \eps^2\,\|\partial_\theta J^\eps (t)  \|_{L^2}^2\right) \,+\,
                  \frac{2\,\tkappa^2}{\eta_{21}}\,\cA (t)\,, 
\end{eqnarray*}
whereas the second term $\cE_{22}(t)$ is evaluated as
\begin{eqnarray*}
 \cE_{22}(t)  &\leq & \tkappa\, \|\rho_{\rm in}^\eps\|_{L^1}\, \|\partial_\theta
                  v^\eps (t) \|_{L^\infty}\,                       \biggl( 
\| \sin\ast\left(\rho^\eps \,-\,\rho  \,-\, \eps\,\partial_\theta J^\eps\right) (t)  \|_{L^\infty}\,+\,
                      \eps\,\|\sin\ast\partial_\theta J^\eps (t)  \|_{L^\infty} \biggr)\,,
  \\[0.9em]
  &\leq & \tkappa\, \|\partial_\theta
                  v^\eps (t) \|_{L^\infty}\,                       \left( 
\| \sin\ast\left(\rho^\eps \,-\,\rho  \,-\, \eps\,\partial_\theta J^\eps\right) (t)  \|_{L^\infty}\,+\, \eps\, \sqrt{\pi} \, \|\partial_\theta J^\eps (t)  \|_{L^2} \right)\,.
\end{eqnarray*}
Hence, using that $H^1(\T)\subset L^\infty(\T)$, we have
$$
\|\partial_\theta v^\eps \|_{L^\infty} \,\leq \, C_{S,1}\,
\|\partial_\theta v^\eps \|_{H^1} \,  \leq\, C_{S,2}\, \| \rho^\eps \,-\,\rho  \,-\, \eps\,\partial_\theta J^\eps\|_{L^2},
$$
where $C_{S,j}$, for $j=1$, $2$ are two positive constants. Therefore, applying the Young's convolution inequality, 
$$
\| \sin\ast\left(\rho^\eps \,-\,\rho  \,-\, \eps\,\partial_\theta
  J^\eps\right)   \|_{L^\infty} \,=\, \| \sin\ast\partial_{\theta}^2
v^\eps \|_{L^\infty}  \,=\, \| \cos\ast\partial_{\theta}
v^\eps \|_{L^\infty} \,\leq\, \sqrt{\pi}\,\|\partial_\theta v^\eps\|_{L^2}\,,
$$
It yields that for any $\eta_{22}>0$,
$$
  \cE_{22}(t) 
          \,\leq\,
\frac{\eta_{22}}{2}\, \|\rho^\eps (t) \,-\,\rho (t)  \,-\, \eps\,\partial_\theta
  J^\eps (t)  \|_{L^2}^2 \,+\,  \frac{\pi\left(\tkappa\,C_{S,2}\right)^2}{\eta_{22}}\,     \left(2\,\cA (t)\,+\,\eps^2\,\|\partial_\theta J^\eps (t)  \|_{L^2}^2\right)\,. 
$$
Gathering the latter estimates on $\cE_{21}$ and $\cE_{22}$, it gives
\begin{eqnarray*}
 \cE_{2}(t) \,\leq\,\frac{\eta_{21}+\eta_{22}}{2}\, \|\rho^\eps (t) \,-\,\rho (t)  \,-\, \eps\,\partial_\theta
                    J^\eps (t)  \|_{L^2}^2  &+&
                                           \left(\frac{\eta_{21}}{2}
                  + \frac{\pi\left(\tkappa\,C_{S,2}\right)^2}{\eta_{22}}\right)\,  
                                           \eps^2\,\|\partial_\theta  J^\eps (t)  \|_{L^2}^2
  \\[0.9em]
  &+&  2\,\tkappa^2\,\left(  \frac{1}{\eta_{21}}\,+\, \frac{\pi\,C_{S,2}^2}{\eta_{22}}\right)\,    \cA (t)\,.
\end{eqnarray*}
Choosing $\eta_1=\eta_{21}=\eta_{22}=\tsigma/3$ on the estimates of
$\cE_1$ and $\cE_2$, we get that there exists a
constant $C>0$, only depending on $\tkappa$ and $\tsigma$, such that
\begin{eqnarray*}
  \cE_{1}(t) \,+\,  \cE_{2}(t)   &\leq& \frac{\tsigma}{2}\, \|\rho^\eps (t) \,-\,\rho (t)  \,-\, \eps\,\partial_\theta
                                        J^\eps (t)  \|_{L^2}^2 \\[0.9em]
                                 &&+\, C\,\left( \cA(t) + \eps^2
                                        \,\|\partial_\theta J^\eps
                                        \|_{L^2}^2 \,+\,  \|f^\eps -
                                        \rho^\eps \cM \|_{L^2_{\cM^{-1}}}^2\right).  
\end{eqnarray*}
Substituting this latter estimate in \eqref{toto} and using
the estimates in \eqref{relative:H} and
$$
\|\partial_\theta  J^\eps \|_{L^2} \,\leq\, \sqrt{\tsigma}\,\|\partial_\theta f^\eps\|_{L^2_{\cM^{-1}}}\,,
$$
with \eqref{res:2}, we  deduce that there exists a constant $C>0$,
only depending on $\tkappa$ and $\tsigma$, such that
$$
\frac{\dD \cA }{\dD t}(t)    \,\leq\, C\,\left( \cA(t)  \,+\,   \|f^\eps_{\rm in} -
                                        \rho^\eps_{\rm in} \cM\|_{L_{\cM^{-1}}^2}^2
  \,\,e^{-\tsigma\, t/(2\,\eps^2)} \,+\, \eps^2  \left(  \|\partial_\theta
   f^\eps_{\rm in}\|_{L^2_{\cM^{-1}}}^2 \,+\,  \|
   f^\eps_{\rm in}\|_{L^2_{\cM^{-1}}}^2  \right)\,e^{C\,t} \right)\,.
$$
Integrating this differential inequality, it yields that there exists a
  constant $C>0$, only depending on $\tkappa$ and $\tsigma$, such that
 $$    
  \cA(t) \,\leq \, \left(\cA(0) \,+\, C \,\left(   \|
   f^\eps_{\rm in}\|_{L^2_{\cM^{-1}}}^2 +\|\partial_\theta
   f^\eps_{\rm in}\|_{L^2_{\cM^{-1}}}^2  \right) \, \eps^2 \right) \,e^{C\,t}\,. 
  $$
  Applying Lemma \ref{lem:4}, we get the second estimate
  \eqref{cv:rho} of Theorem \ref{th:2}.
  $$
\left\| \rho^\eps-\rho\right\|_{H^{-1}}\,\leq\, C\left(
\left\|\rho^\eps_{\rm in} -\rho_{\rm in}\right\|_{H^{-1}} \,+\, \eps
\left(\|f^\eps_{\rm in}\|_{L^2_{\cM^{-1}}}
  \,+\, \|\partial_\theta f^\eps_{\rm in}\|_{L^2_{\cM^{-1}}}\right) \right) \,e^{C\,t}\,.
  $$

\section{Conclusion} \label{sec:Con}
\setcounter{equation}{0}
In this paper, we first studied the stability of a
phase-homogeneous stationary state to the inertial Kuramoto-Sakaguchi
equation. We showed that when the noise intensity
is sufficiently and relatively larger than the coupling strength, the
solutions of the inertial Kuramoto-Sakaguchi equation \eqref{K3-1}-\eqref{K3-2}
converge to the corresponding phase-homogeneous stationary state
exponentially fast in weighted $L^2_\myw$ norm sense. To achieve
this, we employed an energy functional which is equivalent to the
weighted $L^2_\myw$ norm and proved the exponential decaying of
it. Note that there is no smallness assumption on the initial
data. Furthermore, it is notable that we improved the existing results
in \cite{CHXZ}.  Indeed, we proved the convergence for a larger class
of functions. In addition, for the case of sufficiently small or large
coupling strength, that is when coupling strength is near zero or
infinity, we provided smaller lower bound for noise intensity. Finally
when  all oscillators are identical, we investigate a particular
regime corresponding to the long time behavior and the mass $m$ of the
single oscillator converges to zero. This corresponds to the diffusive
limit of   the inertial Kuramoto-Sakaguchi equation for which we prove
error estimate with respect to $m$.

It is worth to mention that the present contribution proposes a simple
proof of two results already
given in \cite{CHXZ} and \cite{HSZ}.  The advantage of our approach is
to present a continuous framework which will be useful for  the design and analysis  of
a fully discrete finite volume scheme for the inertial Kuramoto-Sakaguchi equation \eqref{K3-1}-\eqref{K3-2} written as an hyperbolic system using
Hermite polynomials in velocity \cite{BF22, BF23}. This approach
should allow to  preserve the stationary solution and the weighted
$L^2$ relative energy.

\bibliographystyle{amsplain}

\bibliography{refer}

\end{document}